\documentclass[12pt,twoside,reqno,centertags,draft]{amsart}

  \usepackage{amsmath,amsthm,amsfonts,amssymb}
\begin{document}

\title
{ equations  involving  fractional
Laplacian operator: Compactness  and application}
\date{}
\maketitle

\vspace{ -1\baselineskip}

{\small
\begin{center}
{\sc Shusen Yan}\\
Department of Mathematics,
The  University of New England\\
Armidale NSW 2351,
Australia\\
email:  syan@turing.une.edu.au\\[10pt]
{\sc Jianfu Yang}\\
Department of Mathematics,
Jiangxi Normal University\\
Nanchang, Jiangxi 330022,
P.~R.~China\\
email:  jfyang\_2000@yahoo.com\\[10pt]
and\\[10pt]
 {\sc Xiaohui Yu} \\
Institute for Advanced Study,
Shenzhen University\\
Shenzhen, Guangdong 518060,
P.~R.~China \\
email:  yuxiao\_211@163.com\\[10pt]

\end{center}
}

\renewcommand{\thefootnote}{}
\footnote{AMS Subject Classifications: 35J60,  35J65.} \footnote{Key
words: fractional Laplacian, critical elliptic problem, compactness, infinitely  many solutions.}

\begin{quote}
\footnotesize {\bf Abstract.}
In this paper, we consider the following problem involving  fractional
Laplacian operator:
\begin{equation}\label{eq:0.1}
(-\Delta)^{\alpha} u= |u|^{2^*_\alpha-2-\varepsilon}u + \lambda u\,\,  {\rm in}\,\, \Omega,\quad
u=0 \,\, {\rm on}\, \, \partial\Omega,
\end{equation}
where $\Omega$ is a smooth bounded domain in $\mathbb{R}^N$, $\varepsilon\in [0, 2^*_\alpha-2)$,
$0<\alpha<1,\, 2^*_\alpha = \frac {2N}{N-2\alpha}$. We show that for any sequence of solutions $u_n$
of \eqref{eq:0.1} corresponding to $\varepsilon_n\in [0, 2^*_\alpha-2)$, satisfying $\|u_n\|_{H}\le C$
in the Sobolev space $H$ defined in \eqref{eq:1.1a}, $u_n$ converges strongly in $H$  provided
that $N>6\alpha$  and $\lambda>0$. An application of this compactness result is that
problem \eqref{eq:0.1} possesses infinitely many solutions under the same assumptions.

\end{quote}

\newcommand{\N}{\mathbb{N}}
\newcommand{\R}{\mathbb{R}}
\newcommand{\Z}{\mathbb{Z}}

\newcommand{\cA}{{\mathcal A}}
\newcommand{\cB}{{\mathcal B}}
\newcommand{\cC}{{\mathcal C}}
\newcommand{\cD}{{\mathcal D}}
\newcommand{\cE}{{\mathcal E}}
\newcommand{\cF}{{\mathcal F}}
\newcommand{\cG}{{\mathcal G}}
\newcommand{\cH}{{\mathcal H}}
\newcommand{\cI}{{\mathcal I}}
\newcommand{\cJ}{{\mathcal J}}
\newcommand{\cK}{{\mathcal K}}
\newcommand{\cL}{{\mathcal L}}
\newcommand{\cM}{{\mathcal M}}
\newcommand{\cN}{{\mathcal N}}
\newcommand{\cO}{{\mathcal O}}
\newcommand{\cP}{{\mathcal P}}
\newcommand{\cQ}{{\mathcal Q}}
\newcommand{\cR}{{\mathcal R}}
\newcommand{\cS}{{\mathcal S}}
\newcommand{\cT}{{\mathcal T}}
\newcommand{\cU}{{\mathcal U}}
\newcommand{\cV}{{\mathcal V}}
\newcommand{\cW}{{\mathcal W}}
\newcommand{\cX}{{\mathcal X}}
\newcommand{\cY}{{\mathcal Y}}
\newcommand{\cZ}{{\mathcal Z}}

\newcommand{\abs}[1]{\lvert#1\rvert}
\newcommand{\xabs}[1]{\left\lvert#1\right\rvert}
\newcommand{\norm}[1]{\lVert#1\rVert}

\newcommand{\loc}{\mathrm{loc}}
\newcommand{\p}{\partial}
\newcommand{\h}{\hskip 5mm}
\newcommand{\ti}{\widetilde}
\newcommand{\D}{\Delta}
\newcommand{\e}{\epsilon}
\newcommand{\bs}{\backslash}
\newcommand{\ep}{\emptyset}
\newcommand{\su}{\subset}
\newcommand{\ds}{\displaystyle}
\newcommand{\ld}{\lambda}
\newcommand{\vp}{\varphi}
\newcommand{\wpp}{W_0^{1,\ p}(\Omega)}
\newcommand{\ino}{\int_\Omega}
\newcommand{\bo}{\overline{\Omega}}
\newcommand{\ccc}{\cC_0^1(\bo)}
\newcommand{\iii}{\opint_{D_1}D_i}

\theoremstyle{plain}
\newtheorem{Thm}{Theorem}[section]
\newtheorem{Lem}[Thm]{Lemma}
\newtheorem{Def}[Thm]{Definition}
\newtheorem{Cor}[Thm]{Corollary}
\newtheorem{Prop}[Thm]{Proposition}
\newtheorem{Rem}[Thm]{Remark}
\newtheorem{Ex}[Thm]{Example}

\numberwithin{equation}{section}
\newcommand{\meas}{\rm meas}
\newcommand{\ess}{\rm ess} \newcommand{\esssup}{\rm ess\,sup}
\newcommand{\essinf}{\rm ess\,inf} \newcommand{\spann}{\rm span}
\newcommand{\clos}{\rm clos} \newcommand{\opint}{\rm int}
\newcommand{\conv}{\rm conv} \newcommand{\dist}{\rm dist}
\newcommand{\id}{\rm id} \newcommand{\gen}{\rm gen}
\newcommand{\opdiv}{\rm div}

\vskip 0.2cm \arraycolsep1.5pt
\newtheorem{Lemma}{Lemma}[section]
\newtheorem{Theorem}{Theorem}[section]
\newtheorem{Definition}{Definition}[section]
\newtheorem{Proposition}{Proposition}[section]
\newtheorem{Remark}{Remark}[section]
\newtheorem{Corollary}{Corollary}[section]

\section {Introduction}

\setcounter{equation}{0}

In this paper, we consider the  following  problem with the
fractional Laplacian:
\begin{equation}\label{eq:1.1}
\left\{ \arraycolsep=1.5pt
\begin{array}{lll}
(-\Delta)^{\alpha} u= |u|^{2^*_\alpha-2-\varepsilon}u + \lambda u\,\, & {\rm in}\ \Omega,\\[2mm]
u=0, & {\rm on}\ \ \partial\Omega,
\end{array}
 \right.
\end{equation}
where $\Omega$ is a smooth bounded domain in $\mathbb{R}^N$, $\varepsilon\in [0, 2^*_\alpha-2)$
$\lambda>0$,  $0<\alpha<1$,  and $2^*_\alpha = \frac
{2N}{N-2\alpha}$ is the critical exponent in fractional
 Sobolev inequalities.

In a bounded domain $\Omega\subset \mathbb{R}^N$, we define the operator $(-\Delta)^{\alpha}$ as follows.
Let $\{\lambda_k,\varphi_k\}^\infty_{k=1}$ be the eigenvalues and corresponding eigenfunctions of the
Laplacian operator $-\Delta$ in $\Omega$ with zero Dirichlet boundary values on $\partial\Omega$ normalized by $\|\varphi_k\|_{L^2(\Omega)} = 1$, i.e.
\[
-\Delta \varphi_k = \lambda_k \varphi_k\quad{\rm in}\ \Omega;\quad \varphi_k = 0\quad{\rm on}\ \partial\Omega.
\]
For any $u\in L^2(\Omega)$, we may write $$u = \sum_{k=1}^\infty u_k\varphi_k, \quad{\rm where}\quad u_k = \int_\Omega u\varphi_k\,dx.$$
 We define the space
\begin{equation}\label{eq:1.1a}
H=\{u=\sum_{k=1}^\infty u_k\varphi_k\in L^2(\Omega): \sum_{k=1}^\infty \lambda_k^{\alpha}u_k^2<\infty\},
\end{equation}
which is equipped with the norm
\[
\|u\|_{H} =\bigg(\sum_{k=1}^\infty \lambda_k^{\alpha}u_k^2\bigg)^{\frac 12}.
\]
For any $u\in H$, the fractional Laplacian $(-\Delta)^{\alpha}$ is defined by
\[
(-\Delta)^{\alpha }u = \sum_{k=1}^\infty \lambda_k^{\alpha}u_k\varphi_k.
\]
With this definition, we see that
problem \eqref{eq:1.1} is the Br\'{e}zis-Nirenberg type problem
with the fractional Laplacian. In \cite{BN}, Br\'{e}zis and Nirenberg considered the existence of positive solutions for problem \eqref{eq:1.1} with $\alpha = 1$  and $\varepsilon=0$. Such a problem involves
 the critical Sobolev exponent $2^* = \frac {2N}{N-2}$ for $N\geq 3$, and it is well known that the Sobolev embedding $H^1_0(\Omega)\hookrightarrow L^{2^*}(\Omega)$ is not compact even if $\Omega$ is bounded. Hence, the associated functional of problem \eqref{eq:1.1} does not satisfy the Palais-Smale condition, and critical point theory cannot be applied
directly to find solutions of the problem. However, it is found in \cite{BN} that the functional satisfies the  $(PS)_c$ condition for $c\in (0, \frac 1N S^{\frac N2})$, where $S$ is the best Sobolev constant and $\frac 1N S^{\frac N2}$ is the least level at which the Palais-Smale condition fails. So a positive solution can be found if the mountain pass value corresponding to problem \eqref{eq:1.1} is strictly less than $\frac 1N S^{\frac N2}$. In \cite{L}, a concentration-compactness principle was developed to treat non-compact critical variational problems.
In the study of the existence of multiple solutions for critical problems, to retain the compactness, it is necessary to have a full description of energy levels at which the associated functional does not satisfy the Palais-Smale condition.
A global compactness result is found in \cite{ST},
 which describes precisely the obstacles
of the compactness for critical semilinear elliptic problems. This compactness result
shows that above certain energy level, it is impossible to prove  the Palais-Smale condition.
 For this reason, to obtain many solutions for the critical problem, it is
 essential to find a condition that can replace the standard Palais-Smale condition.

  In \cite{DS}, Devillanova and Solimini considered  \eqref{eq:1.1} with $\alpha=1$. They
started by considering any sequence of solutions $u_n$
of \eqref{eq:1.1} corresponding to $\varepsilon_n>0$, $\varepsilon_n\to 0$, satisfying $\|u_n\|_{H}\le C$
in the Sobolev space $H$ defined in \eqref{eq:1.1a}.
By analyzing the bubbling behaviors of $u_n$, they are able to show that $u_n$ converges strongly to a solution of the critical problem  in $H$ if $N>7$ and $\lambda>0$. A consequence of this compactness result is that \eqref{eq:1.1} with $\alpha=1$ is that \eqref{eq:1.1} with $\alpha=1$ and $\varepsilon=0$  has infinitely many solutions.
 So, we see that the compactness of the solutions set for  \eqref{eq:1.1} can be used to replace
 the Palais-Smale condition in the critical point theories.

Let us point out that the same  idea was
used  in \cite{CPY}, \cite{CY} and \cite{YY}
to study other problems involving critical exponents,
though the methods used in \cite{CPY,CY,YY}
to obtain the estimates are different from those
in \cite{DS}.

Problems with the fractional Laplacian  have been
extensively studied recently. See
for example \cite{BCPa, BCP, CabS, CT, CDDS, CSS, CS, JLX, S, SV, T, TX}. In particular, the
Br\'{e}zis-Nirenberg type problem was discussed in \cite{T} for the
special case $\alpha = \frac 12$,  and in \cite{BCP} for the general
case, $0<\alpha<1$, where existence of one positive solution was
proved.  To use the idea in \cite{BN} to prove the existence of one
positive solution for the fractional Laplacian, the authors in
\cite{BCP, T} used the following results in \cite{CS} (see also
\cite{BCPa}):  for any $u\in H$, the solution $v\in
H^1_{0,L}(\mathcal{C}_\Omega)$ of the problem
\begin{equation}\label{eq:1.2}
\left\{ \arraycolsep=1.5pt
\begin{array}{ll}
-{\rm div}(y^{1-2\alpha}\nabla v) = 0, & \text{in}\; \mathcal{C}_{\Omega}=\Omega\times(0,\infty),\\[1mm]
v=0,& \text{on}\;
\partial_L\mathcal{C}_{\Omega}=\partial\Omega\times(0,\infty),\\[1mm]
v = u , &\text{on}\;  \Omega\times\{0\},
 \end{array}
 \right.
 \end{equation}
satisfies  $$-\lim_{y\to 0^+}k_\alpha y^{1-2\alpha}\frac {\partial
v}{\partial y} = (-\Delta)^{\alpha}u,$$
 where  we use  $(x,y)= (x_1,\cdots,x_N,y)\in \mathbb{R}^{N+1}$, and

\begin{equation}\label{1-18-5}
H^1_{0,L}(\mathcal{C}_\Omega) = \{v\in L^2(\mathcal{C}_\Omega): v =
0 \,\,{\rm on}\,\, \partial_L\mathcal{C}_\Omega,\
\int_{\mathcal{C}_\Omega}y^{1-2\alpha}|\nabla v|^2\,dxdy<\infty\}.
\end{equation}
Therefore,  the nonlocal problem (\ref{eq:1.1}) can be reformulated
to the following  local problem:
\begin{equation}\label{eq:1.3}
\left\{ \arraycolsep=1.5pt
\begin{array}{lll}
-{\rm div}(y^{1-2\alpha}\nabla v) = 0,  & \text{in}\;\mathcal{C}_{\Omega},\\[1mm]
v=0,
&\text{on}\;\partial_L\mathcal{C}_{\Omega},\\[1mm]
 y^{1-2\alpha}\frac {\partial v}{\partial \nu} = |v(x,0)|^{2^*_\alpha-2-\varepsilon}v(x,0)+ \lambda v(x,0),  &\text{on}\;\Omega\times\{0\},
 \end{array}
 \right.
 \end{equation}
where $\frac {\partial }{\partial \nu}$ is the outward normal
derivative of $\partial \mathcal{C}_{\Omega}$. Hence, critical
points of the functional
\begin{equation}\label{eq:1.4}
I_\varepsilon(v) = \frac 12 \int_{\mathcal{C}_\Omega}y^{1-2\alpha}|\nabla v|^2\,dxdy -\frac 1{2^*_\alpha-\varepsilon}\int_{\Omega\times\{0\}} |v|^{2^*_\alpha-\varepsilon}\,dx -\frac \lambda 2\int_{\Omega\times\{0\}} |v|^{2}\,dx
\end{equation}
defined on $H^1_{0,L}(\mathcal{C}_\Omega)$ correspond to solutions
of (\ref{eq:1.3}). A solution at the mountain pass level of the
functional $I(u)$ was found in \cite{BCP, T}. On the other hand, it is
easy to  show by using the Pohozaev type identity that the problem
\[
(-\Delta)^{\alpha} u= |u|^{p-1}u \quad {\rm in}\quad \Omega,\quad
u=0\quad  {\rm on}\quad  \partial\Omega
\]
has no nontrivial solution if $p+1\geq \frac{2N}{N-2\alpha}$ and $\Omega$ is star-shaped.

In this paper, we  will investigate the existence of infinitely many
solutions for
 problem \eqref{eq:1.1} by finding critical points of
the functional $I(u)$. Since the problem is critical, the functional
$I(u)$ does not satisfy the Palais-Smale condition.  Thus the mini-max
theorems can not be applied directly to obtain infinitely many
solutions for \eqref{eq:1.1}.  So we follow the idea in \cite{DS} to
consider the subcritical problem
\begin{equation}\label{eq:1.5}
\left\{ \arraycolsep=1.5pt
\begin{array}{lll}
div(y^{1-2\alpha}\nabla v) = 0, & {\rm in}\ \mathcal{C}_\Omega,\\[2mm]
v = 0,   & {\rm on}\ \ \partial_L \mathcal{C}_\Omega,\\[2mm]
y^{1-2\alpha}\frac {\partial v}{\partial y} = - |v(x,0)|^{p_n-2}v(x,0) - \lambda v(x,0),   & {\rm on}\ \
\Omega\times \{0\}.
 \end{array}
 \right.
 \end{equation}
where $p_n = 2_\alpha^* - \varepsilon_n$ with $\varepsilon_n\to 0$.

The main result of this paper is the following.

\begin{Theorem}\label{thm:1.1}

Suppose $N> 6\alpha$,
 then for any $v_{n}$, which is a
solution of \eqref{eq:1.5} satisfying
$\|v_n\|_{H^1_{0,L}(\mathcal{C}_\Omega)}$ $\leq C$ for some constant
independent of $n$, $v_n$ converges strongly in
$H^1_{0,L}(\mathcal{C}_\Omega)$ as $n\to +\infty$.
\end{Theorem}

Theorem~\ref{thm:1.1} is a special compactness result. It shows that
although $I(u)$ does not satisfy the Palais--Smale condition, for a
special Palais--smale sequence, which is  solutions of the perturbed
problem \eqref{eq:1.5}, it does converge strongly in
$H^1_{0,L}(\mathcal{C}_\Omega)$. It is well known now \cite{CDS,CPY} that this weak
compactness leads to the following existence result:

\begin{Theorem}\label{thm:1.2}
 If $N> 6\alpha$, then \eqref{eq:1.1} with $\varepsilon=0$ has infinitely many solutions.
\end{Theorem}

The main difficulty in the study of \eqref{eq:1.5} is that we need to
carry out the boundary estimates. This is different from
 the Dirichlet problems studied in \cite{CDS,CPY,CY,DS,YY}, which mainly involve the
interior estimates.

This paper is organized as follows.  In section~2,  we will state a decomposition
result for the solutions of the perturbed problem~\eqref{eq:1.5}.  In section~3,
we obtain some integral estimates which captures the possible bubbling behavior of the
  solutions of \eqref{eq:1.5}.  To prove such estimates, we need to study a linear
problem. This part is of independent interest. So we put it in Appendix~A.  Section~4
contains the estimates for solutions of  \eqref{eq:1.5} in the  region which does
not contain any blow up point, but
is close to some blow up point.  The main result is proved in section~5 by using
the local Pohozaev identity, together with the estimates in section~4.  In Appendix~B,
we prove a decay estimate for solutions of a problem in half space involving the
fractional critical Sobolev exponent.

Throughout this paper,  we use $\mathcal {B}_r(z)$ to denote the ball in $\mathbb{R}^{
N+1}$, centered at $z\in \mathbb{R}^{N+1}$ with radius $r$. We also use $X=(x,y)$ to
denote a point  in $\mathbb{R}^{N+1}$,  and for any set $D\in \mathbb{R}^{N}$,

\begin{equation}\label{1-1-9}
\mathcal{C}_{D}= D\times(0,\infty)\subset \mathbb{R}^{N+1},\quad
\partial_L \mathcal{C}_{D}=\partial D\times (0, +\infty).
\end{equation}

\section  {Preliminaries}

\setcounter{equation}{0}

Let $\Omega$ be a smooth bounded domain in $\mathbb{R}^N$ and $0<\alpha<1$. The space $H^\alpha(\Omega)$ is defined as the subset of $L^2(\Omega)$ such that for  $u\in L^2(\Omega)$, the norm
\[
\|u\|_{H^\alpha(\Omega)} = \|u\|_{L^2(\Omega)} + \biggl(\int_\Omega\int_\Omega\frac{|u(x) - u(\tilde x)|^2}{|x-\tilde x|^{N+2\alpha}}\,dxd\tilde x\biggr)^{\frac12}
\]
is finite.  Let $H^\alpha_0(\Omega)$ be the closure of $C^\infty_0(\Omega)$ with respect to the norm $\|\cdot\|_{H^\alpha(\Omega)}$.
It is known from \cite{LM} that for $0<\alpha\leq \frac 12$, $H^\alpha_0(\Omega) = H^\alpha(\Omega)$; for $\frac 12<\alpha< 1$, $H^\alpha_0(\Omega)\varsubsetneq H^\alpha(\Omega)$.

The space $H$ defined in \eqref{eq:1.1a} is the interpolation
space $(H^2_0(\Omega), L^2(\Omega))_{\alpha,2}$, see \cite{A, LM,Tar}. It was shown
in \cite{LM} that $(H^2_0(\Omega), L^2(\Omega))_{\alpha,2} = H^\alpha_0(\Omega)$ if $0<\alpha<1$ and $\alpha\not=\frac 12$; while $(H^2_0(\Omega), L^2(\Omega))_{\frac 12 ,2} = H_{00}^{\frac12}(\Omega)$, where
\[
H_{00}^{\frac12}(\Omega) =\{u\in H^{\frac 12}(\Omega):\int_\Omega\frac{u^2(x)}{d(x)}\,dx<\infty\},
\]
and $d(x) = dist(x,\partial\Omega)$ for all $x\in\Omega$. We know from \cite{BCP}, see also \cite{CDDS}, that for any $u\in H^\alpha_0(\Omega)$, let $v\in H^1_{0,L}(\mathcal{C}_\Omega)$ be the extension of $u$ defined in \eqref{eq:1.2}, then the mapping $u\to v$ is an isometry between $H^\alpha_0(\Omega)$ and $H^1_{0,L}(\mathcal{C}_\Omega)$. That is
\[
\|v\|_{H^1_{0,L}(\mathcal{C}_\Omega)} = \|u\|_{H^\alpha_0(\Omega)}\quad {\rm for}\quad u\in H^\alpha_0(\Omega).
\]

For any function  $W$ defined on $\mathbb{R}^{N+1}$,  $x\in \mathbb{R}^N$, $\sigma>0$, we define

\begin{equation}\label{2-1-9}
\rho_{x,\sigma}( W) = \sigma^{\frac{N-2\alpha}2}W \bigl(\sigma(\cdot- (x,0))\bigr).
\end{equation}

It
is now  standard to prove the following decomposition result.

\begin{Proposition}\label{prop:2.1}
Let $\{v_n\}\subset H^1_{0,L}(\mathcal{C}_\Omega)$ be a sequence of solutions of
\eqref{eq:1.5} satisfying
$\|v_n\|_{H^1_{0,L}(\mathcal{C}_\Omega)}\leq C$.  Then, there exist
a solution $v_0\in H^1_{0,L}(\mathcal{C}_\Omega)$ of (\ref{eq:1.3}),
a finite sequence $\{W^j\}_{j=1}^k\subset H^1_{0,L}(\mathbb{R}^N)$, which are
solutions of
\begin{equation}\label{eq:2.3}
\left\{ \arraycolsep=1.5pt
\begin{array}{lll}
div(y^{1-2\alpha}\nabla v) = 0, & {\rm in}\ \mathbb{R}^{N+1}_+,\\[2mm]
y^{1-2\alpha}\frac {\partial v}{\partial y} = - \beta_j
|v(x,0)|^{2^*_\alpha-2}v(x,0),   & {\rm in}\ \ \mathbb{R}^N,
 \end{array}
 \right.
 \end{equation}
where $\beta_j\in (0, 1]$ is some constant, and sequences
$\{x_n^j\}_{j=1}^k$, $\{\sigma_n^j\}_{j=1}^k$ satisfying
$\sigma_n^j>0$, $x_n^j\in \Omega$ and as $n\to +\infty$,

\begin{equation}\label{19-30-8}
\sigma_n^jdist(x_n^j,\partial\Omega)\to\infty,\,\,\frac{\sigma_n^j}{\sigma_n^i}+\frac{\sigma_n^i}{\sigma_n^j}
+\sigma_n^i\sigma_n^j |x_n^i-x_n^j|^2\to +\infty,\quad i\ne j,
 \end{equation}

\begin{equation}\label{20-30-8}
\|v_n - v_0 - \sum_{j=1}^k \rho_{x_n^j,\sigma_n^j}( W^j)\|_{H^1_{0,L}(\mathbb{R}^N)} \to 0.
 \end{equation}

\end{Proposition}

\section {Integral Estimates }

\setcounter{equation}{0}

\bigskip

To prove Theorem~\ref{thm:1.1}, we need to prove  that the bubbles
$\rho_{x_n^j,\sigma_n^j} (W^j)$ do not appear in the decomposition
\eqref{20-30-8}.

Similar to \cite{DS}, we introduce the following norm. Let $q_1,
q_2\in (2, \infty)$ be such that $q_2< 2^*_\alpha< q_1,\, \beta> 0$
and $\sigma>0$. We consider the following inequalities
\begin{equation}\label{eq:3.2}
\left\{ \arraycolsep=1.5pt
\begin{array}{lll}
\|u_1\|_{q_1}\leq \beta, \\[2mm]
\|u_2\|_{q_2}\leq \beta \sigma^{\frac N {2^*_\alpha}-\frac N{q_2}}
 \end{array}
 \right.
 \end{equation}
and define the norm
\begin{equation}\label{eq:3.3}
\|u\|_{q_1,q_2,\sigma} = \inf\{\beta>0: {\rm there\, exist}\, u_1, u_2 \,{\rm such \,that} \,\eqref{eq:3.2}\, {\rm holds\, and}\, |u|\leq u_1 + u_2\}.
 \end{equation}

Denote

\[
\sigma_n =\min_{1\le j\le k} \sigma_n^j.
\]

 In this section, we will prove the following result.

\begin{Proposition}\label{prop:3.1}
Let $v_n$ be a solution of \eqref{eq:1.5}.  For any $q_1, q_2\in
(\frac{N}{N-2\alpha},$ $ +\infty)$, $q_2< 2_\alpha^*<q_1$, there is
a constant $C>0$, depending only on $q_1$ and $q_2$, such that

\begin{equation}\label{eq:3.16}
\|v_n\|_{q_1,q_2,\sigma_n}\leq C.
\end{equation}
\end{Proposition}

 To
prove Proposition~\ref{prop:3.1},  it is convenient to consider the
following problem. Let $D$ be a bounded domain such that
$\Omega\subset\subset D$ and let $v_n(x,0) = 0$ in
$D\setminus\Omega$.  We choose $A>0$ large enough so that
\[
\big||t|^{p_n-2}t+ \lambda t\big|\leq 2|t|^{2^*_\alpha-1} + A,\quad \forall\; t\in \mathbb{R}.
\]
Solving
\begin{equation}\label{eq:3.3a}
\left\{ \arraycolsep=1.5pt
\begin{array}{lll}
{\rm div}(y^{1-2\alpha}\nabla w) = 0,  &\text{in}\;\mathcal{C}_{D},\\[1mm]
w=0,
&\text{on}\;\partial_L\mathcal{C}_{D},\\[1mm]
 y^{1-2\alpha}\frac {\partial w}{\partial \nu} = 2|v_n(x,0)|^{2^*_\alpha-1} + A,  &
\text{on}\; D\times\{0\},
 \end{array}
 \right.
 \end{equation}
we obtain a sequence of solutions $\{w_n\}$ with $w_n\geq 0$.  By the choice of $D$ and $A$, we find
\begin{equation}\label{eq:3.3b}
\left\{ \arraycolsep=1.5pt
\begin{array}{lll}
{\rm div}(y^{1-2\alpha}\nabla (w_n\pm v_n)) = 0,  & \text{in}\;\mathcal{C}_{\Omega},\\[1mm]
w_n\pm v_n\ge 0,
&\text{on}\;\partial_L\mathcal{C}_{\Omega},\\[1mm]
 y^{1-2\alpha}\frac {\partial (w_n\pm v_n)}{\partial \nu} \geq 0,  &
\text{on}\;\Omega\times\{0\}.
 \end{array}
 \right.
 \end{equation}
Multiplying  \eqref{eq:3.3b} by $(w_n\pm v_n)^-$ and integrating by
part, we see that

\[
|v_n|\leq w_n,\quad \text{in}\;\mathcal{C}_{\Omega}.
\]
Hence, it is sufficient to estimate $w_n$ in $\mathcal{C}_D$.

\begin{Lemma}\label{lem:3.1}

Let $w\in H^1_{0,L}(\mathcal{C}_D)$ be a solution of
\begin{equation}\label{1-31-8}
\left\{ \arraycolsep=1.5pt
\begin{array}{lll}
div(y^{1-2\alpha}\nabla w)=0\;\;    &{\rm in}\quad \mathcal{ C}_D, \\[2mm]
 w=0  & {\rm on}\quad \partial_L  \mathcal{ C}_D , \\[2mm]
 y^{1-2\alpha}\frac{\partial w}{\partial \nu}=a(x)v  & {\rm on}\quad
D\times\{0\}, \\[2mm]
 \end{array}
 \right.
 \end{equation}
where $a\in L^{\frac N{2\alpha}}(D), v\in C^\beta(D)$  and $a, v\geq 0$.
 For any  $q_1, q_2\in (\frac{N}{N-2\alpha}, +\infty)$, $q_2< 2_\alpha^*<q_1$, there exists $C = C(N,q_1,q_2)>0$,
such that
\begin{equation}\label{eq:3.5}
\|w(\cdot,0)\|_{q_1,q_2,\sigma}\leq C\|a\|_{L^{\frac N{2\alpha}}(D)}\|v\|_{q_1,q_2,\sigma}.
\end{equation}
\end{Lemma}

\begin{proof} For any $\varepsilon>0$ small and $\sigma>0$ fixed, let $v_1\geq 0$ and $v_2\geq 0$ be functions such that $|v|\leq v_1 + v_2$ and satisfying \eqref{eq:3.2} with $\beta = \|v\|_{q_1,q_2,\sigma} + \varepsilon$.  For $i= 1,2$,  consider
\begin{equation}\label{eq:3.6}
\left\{ \arraycolsep=1.5pt
\begin{array}{lll}
div(y^{1-2\alpha}\nabla w_i)= 0 \;\;   &{\rm in}\quad  \mathcal{C}_D, \\[2mm]
 w_i=0  & {\rm on}\quad \partial_L\mathcal{C}_D, \\[2mm]
 y^{1-2\alpha}\frac{\partial w_i}{\partial \nu}=a(x)v_i  & {\rm on}\quad
D\times\{0\}.\\[2mm]
 \end{array}
 \right.
 \end{equation}
By Corollary~\ref{lem:2.1},
\begin{equation}\label{eq:3.7}
\|w_i(\cdot,0)\|_{L^{q_i}(D)}\leq C\|a\|_{L^{\frac N{2\alpha}}(D)}\|v_i\|_{L^{q_i}(D)}, \,\, i = 1, 2.
\end{equation}

On the other hand, it follows from the comparison theorem   that

\[
0\le w\le w_1 +w_2,
\]
since $|v|\le v_1 +v_2$.
Thus we complete the proof.

\end{proof}

\begin{Lemma}\label{lem:3.2}

Let $w>0$ be the solution of
\begin{equation}\label{2-31-8}
\left\{ \arraycolsep=1.5pt
\begin{array}{lll}
{\rm div}(y^{1-2\alpha}\nabla w) = 0, \ & \text{in}\;\mathcal{C}_{D},\\[1mm]
w=0,\ \
&\text{on}\;\partial_L\mathcal{C}_{D},\\[1mm]
 y^{1-2\alpha}\frac {\partial w}{\partial \nu} = 2|v(x,0)|^{2^*_\alpha-1} + A, \ \ &
\text{on}\;D\times\{0\},
 \end{array}
 \right.
 \end{equation}
where $v\in C^\beta(D)$ is a nonnegative function.
Suppose $p_1,p_2\in (\frac {N+2\alpha}{N-2\alpha}, \frac N{2\alpha}\frac {N+2\alpha}{N-2\alpha})$ and
$p_2<2^*_\alpha<p_1$. Let $q_1,q_2$ be determined by
\begin{equation}\label{eq:3.11}
\frac 1{q_i}=\frac {N+2\alpha}{N-2\alpha}\frac 1{p_i}-\frac {2\alpha}N, \,\, i = 1, 2.
 \end{equation}
Then,  there exists a constant $C=C(N,p_1,p_2,\Omega)>0$ such that for
any $\sigma>0$, it holds
$$
\|w(\cdot,0)\|_{q_1,q_2,\sigma}\leq C\big(\|v\|_{p_1,p_2,\sigma}^{\frac
{N+2\alpha}{N-2\alpha}}+1\big).
$$
\end{Lemma}

\begin{proof}
Choose  $v_1\ge 0$ and $v_2\ge 0$, with  $|v|\le v_1+v_2$ and

$$
\|v_1\|_{L^{p_1}(D)}\leq (\|v\|_{p_1,p_2,\sigma}+\varepsilon), \quad\|v_2\|_{L^{p_2}(D)}\leq \sigma^{\frac N{2^*_\alpha}-\frac
N{p_2}}( \|v\|_{p_1,p_2,\sigma}+\varepsilon).
$$
Now we consider the following problems

\begin{equation}\label{eq:3.12}
\left\{ \arraycolsep=1.5pt
\begin{array}{lll}
- div(y^{1-2\alpha}\nabla w_1) = 0    &{\rm in}\quad  \mathcal{C}_D, \\[2mm]
 w_1=0  & {\rm on}\quad \partial_L\mathcal{C}_D, \\[2mm]
 y^{1-2\alpha}\frac{\partial w_1}{\partial \nu}= 2^{\frac {4\alpha}{N-2\alpha}} v_1^{\frac {N+2\alpha}{N-2\alpha}} + A\quad  & {\rm on}\quad
D\times\{0\},\\[2mm]
 \end{array}
 \right.
 \end{equation}
and

\begin{equation}\label{eq:3.13}
\left\{ \arraycolsep=1.5pt
\begin{array}{lll}
- div(y^{1-2\alpha}\nabla w_2)= 0    &{\rm in}\quad  \mathcal{C}_D, \\[2mm]
 w_2 = 0  & {\rm on}\quad \partial_L\mathcal{C}_D, \\[2mm]
 y^{1-2\alpha}\frac{\partial w_2}{\partial \nu}= 2^{\frac {4\alpha}{N-2\alpha}} v_2^{\frac {N+2\alpha}{N-2\alpha}}\quad  & {\rm on}\quad
D\times\{0\}.\\[2mm]
 \end{array}
 \right.
 \end{equation}
Since
$$
|v|^{\frac {N+2\alpha}{N-2\alpha}}\leq 2^{\frac {4\alpha}{N-2\alpha}}v_1^{\frac {N+2\alpha}{N-2\alpha}}+2^{\frac {4\alpha}{N-2\alpha}}v_2^{\frac {N+2\alpha}{N-2\alpha}},
$$
by comparison,
 $0\le w\leq w_1 + w_2$. Hence, we need to estimate $\|w_1(\cdot,0)\|_{L^{q_1}(D)}$ and $\|w_2(\cdot,0)\|_{L^{q_2}(D)}$. Since  $1<p_i\frac {N-2\alpha}{N+2\alpha}<\frac N{2\alpha}$, by Proposition~\ref{lem:2.2},
\begin{eqnarray*}
&&\|w_1(\cdot,0)\|_{L^{q_1}(D)}\\
&\leq& C(N,p_1)\|v_1^{\frac{N+2\alpha}{N-2\alpha}}+A
\|_{L^{p_1\frac {N-2\alpha}{N+2\alpha}}(D)}\\
&\leq& C(N,p_1)\big(\|v_1\|_{L^{p_1}(D)}^{\frac{N+2\alpha}{N-2\alpha}}+A|D|^{{\frac
1{p_1}\frac {N+2\alpha}{N-2\alpha}}}\big)\\&\leq&
C(N,p_1,D)\big((\|v\|_{p_1,p_2,\sigma}+\varepsilon)^{\frac {N+2\alpha}{N-2\alpha}}+1\big).
\end{eqnarray*}
Similarly, we have
\[
\|w_2(\cdot,0)\|_{L^{q_2}(D)}
\leq C\|v_2\|_{L^{p_2}D)}^{\frac{N+2\alpha}{N-2\alpha}}
\leq C( \|v\|_{p_1,p_2,\sigma}+\varepsilon)^{\frac{N+2\alpha}{N-2\alpha}}\sigma^{(\frac N{2^*_\alpha}-\frac
N{p_2})\frac{N+2\alpha}{N-2\alpha}}.
\]
Since
$$
\big(\frac N{2^*_\alpha}-\frac
N{p_2}\big)\frac{N+2\alpha}{N-2\alpha}=\frac N{2^*_\alpha}-\frac
N{q_2},
$$
$w_1, w_2$ satisfies \eqref{eq:3.2} with
$\alpha=C\big((\|v\|_{p_1,p_2,\sigma}+\varepsilon)^{\frac
{N+2\alpha}{N-2\alpha}} + 1\big)$. The proof is completed by letting $\varepsilon\to 0$.

\end{proof}

\begin{Lemma}\label{l1-31}
Let $w_n$ be a solution of \eqref{eq:3.3a}.  There are constants  $C>0$,  $q_1, q_2\in
(\frac{N}{N-2\alpha},$ $ +\infty)$, $q_2< 2_\alpha^*<q_1$,  such that

\begin{equation}
\|w_n\|_{q_1,q_2,\sigma_n}\leq C.
\end{equation}
\end{Lemma}

\begin{proof} Since $\{\|v_n\|_{H^1_{0,L}(\mathcal{C}_\Omega)}\}$ is uniformly bounded, we may assume $v_n\rightharpoonup v_0$. By Proposition \ref{prop:2.1}, we may write $v_n=v_0+v_{n,1}+v_{n,2}$, where
\[
v_{n,1}(x,y) = \sum_{j=1}^k\rho_{x_n^j,\sigma_n^j}(W_j)
\]
$v_{n,2} = v_n - v_0- v_{n,1}$. Let $a_0=C |v_{0}|^{\frac
{4\alpha}{N-2\alpha}}$ and $a_i=C |v_{n,i}|^{\frac
{4\alpha}{N-2\alpha}}$, $i=1,2$ for $C>0$ large.

Denote by $w=G(v)$ the solution of the following problem
\begin{equation}\label{eq:3.18}
\left\{ \arraycolsep=1.5pt
\begin{array}{lll}
- div(y^{1-2\alpha}\nabla w) = 0\quad    &{\rm in}\quad  \mathcal{C}_D, \\[2mm]
w = 0  & {\rm on}\quad \partial_L\mathcal{C}_D, \\[2mm]
 y^{1-2\alpha}\frac{\partial w}{\partial \nu}= v \quad  & {\rm on}\quad
D\times\{0\}.\\[2mm]
 \end{array}
 \right.
 \end{equation}
By the comparison theorem,
\[
w_n\leq G(a_0(\cdot,0)|v_n(\cdot,0)| + A) + G(a_1(\cdot,0)|v_n(\cdot,0)|) + G(a_2(\cdot,0)|v_n(\cdot,0)|).
\]

Note that $v_0\in L^\infty(\Omega)$.  So $a_0\in L^\infty(D)$.
 Taking $\frac{2N}{N+2\alpha}<p<2^*_\alpha$, since $N>2\alpha$, we have $\frac N{2\alpha}>2^*_\alpha$, and then $q_1:=\frac{Np}{N-2\alpha p}> 2^*_\alpha$. By Proposition~\ref{lem:2.2} and H\"{o}lder's inequality,
\[
\begin{split}
&\|G(a_0(\cdot,0)|v_n(\cdot,0)|+A)(\cdot,0)\|_{L^{q_1}(D)}\\
&\leq C\|v_n(\cdot,0)\|_{L^p(D)}+C\le
 C\|v_n(\cdot,0)\|_{L^{2^*_\alpha}(D)}+C\le C.
\end{split}
\]
This implies that for any $q_2< 2^*_\alpha$,
\[
\|G(a_0(\cdot,0)|v_n(\cdot,0)|+A)(\cdot,0)\|_{q_1,q_2,\sigma_n}\leq \|G(a_0(\cdot,0)
|v_n(\cdot,0)|+A)(\cdot,0)\|_{L^{q_1}(D)}\leq C.
\]

To estimate   $G(a_1(\cdot,0)|v_n(\cdot,0)|)(\cdot,0)$,
we choose $r$ such that $\frac N{4\alpha}<r< \frac N{2\alpha}$ and $\frac 1{q_2} = \frac 1r + \frac 1{2^*_\alpha} - \frac {2\alpha}N$, we have  $\frac{2N}{N+2\alpha}<q_2<2^*_\alpha$.
By Corollary~\ref{lem:2.3},
\[
\|G(a_1(\cdot,0)|v_n(\cdot,0)|)(\cdot,0)\|_{L^{q_2}(D)}\leq C\|a_1(\cdot,0)\|_{L^{r}(\Omega)}\|v_n(\cdot,0)\|_{L^{2^*_\alpha}(\Omega)}.
\]
Noting that $\frac {N-2\alpha r}r = (\frac 1{q_2} - \frac 1{2^*_\alpha})N$, we find
\[
\|a_1(\cdot,0)\|_{L^{r}(\Omega)} \leq \sum_{j=1}^k(\sigma_n^j)^{-\frac{N-2\alpha r}r}\bigg(\int_{\mathbb{R}^N}|W^j|^{\frac{4r\alpha}{N-2\alpha}}\,dx\bigg)^{\frac 1r} \leq C\sigma_n^{\frac N{2^*_\alpha}-\frac N{q_2}},
\]
since, by Proposition~\ref{p1-30},

\[
|W^j|^{\frac{4r\alpha}{N-2\alpha}}\le \frac{C}{ (1+|X|)^{4r\alpha}}
\]
and  $4r\alpha>N$.
Therefore,
\[
\|G(a_1(\cdot,0)w_n(\cdot,0))(\cdot,0)\|_{q_1,q_2,\sigma_n}
\leq \|G(a_1(\cdot,0)w_n(\cdot,0))(\cdot,0)\|_{L^{q_2}(\Omega)}\sigma_n^{\frac N{q_2}-\frac N{2^*_\alpha}}\leq C.
\]
Using Lemma \ref{lem:3.1}, we deduce
\[
\begin{split}
&\|G(a_2(\cdot,0)|v_n(\cdot,0)|)(\cdot,0)\|_{q_1,q_2,\sigma_n}\\
&\leq \|a_2(\cdot,0)\|_{L^{\frac N{2\alpha}}(\Omega)}
\|v_n(\cdot,0)\|_{q_1,q_2,\sigma_n}
\leq \frac12\|w_n(\cdot,0)\|_{q_1,q_2,\sigma_n}.\\
\end{split}
\]
Consequently,
\[
\begin{split}
&\|w_n(\cdot,0)\|_{q_1,q_2,\sigma_n}\\
&\leq 2\|G(a_0(\cdot,0)w_n(\cdot,0))(\cdot,0)\|_{q_1,q_2,\sigma_n} + 2\|G(a_1(\cdot,0)w_n(\cdot,0))(\cdot,0)\|_{q_1,q_2,\sigma_n}\\
&\leq C.\\
\end{split}
\]
The proof is complete.
\end{proof}

  \begin{proof}[Proof of Proposition~\ref{prop:3.1}]

Since $|v_n|\leq w_n$, by Lemmas~\ref{lem:3.2}  and \ref{l1-31}, for
any constants $q_1, q_2\in (\frac{N}{N-2\alpha},$ $ +\infty)$, $q_2<
2_\alpha^*<q_1$,  it holds

\begin{equation}\label{10-31-8}
\|v_n\|_{q_1,q_2,\sigma_n}\le  \|w_n\|_{q_1,q_2,\sigma_n}\leq C.
\end{equation}
So the result follows.
\end{proof}

\section {Estimates on safe regions}

\setcounter{equation}{0}

\bigskip

Since $\|v_n\|_E$ is uniformly bounded in $n$, the number of the bubble of $v_n$ is also uniformly bounded in $n$, and we can find a constant $\bar C>0$, independent of $n$, such that the region
\[
\mathcal{A}_n^1=\{X=(x,y):\; X\in \bigg(\mathcal{B}_{(\bar C+5)\sigma_n^{-\frac
12}}(x_n,0)\setminus \mathcal{B}_{\bar C\sigma_n^{-\frac 12}}(x_n,0)\bigg)\cap \mathcal{C}_\Omega\}
\]
does not contain any concentration point of $v_n$ for any $n$, where $\mathcal{B}_r(z)$ is the ball in $\mathbb{R}^{N+1}$ centered at $z$ with the radius $r$. We call $\mathcal{A}_n^1$ safe region.
Let
$$\mathcal A_n^2=\{X:  \;X\in\bigg(\mathcal{B}_{(\bar C+4)\sigma_n^{-\frac
12}}(x_n,0)\setminus \mathcal{B}_{(\bar C+1)\sigma_n^{-\frac 12}}(x_n,0)\bigg)\cap \mathcal{C}_\Omega\}
$$
and
$$
\mathcal A_n^3=\{X:\; X\in\bigg(\mathcal{B}_{(\bar
C+3)\sigma_n^{-\frac 12}}(x_n,0)\setminus \mathcal{B}_{(\bar
C+2)\sigma_n^{-\frac 12}}(x_n,0)\bigg)\cap \mathcal{C}_\Omega\}.
$$

In this section, we will prove the following result.

\begin{Proposition}\label{prop:4.1}
 There is a constant $C>0$, independent of $n$, such that
\begin{equation}\label{eq:4.2}
\bigg(\int_{\mathcal{A}_n^2}y^{1-2\alpha}|v_n|^p\,dxdy\bigg)^{\frac 1p}\leq  C\sigma_n^{-\frac{N+2-2\alpha}{2p}}
\end{equation}
and
\begin{equation}\label{eq:4.2a}
\int_{\mathcal{A}_n^2\cap\{y=0\}}|v_n|^p\leq C\sigma_n^{-\frac{N}{2}}
\end{equation}
for any $p\geq 1$.
\end{Proposition}

To prove Proposition \ref{prop:4.1}, we need the following lemmas.
\begin{Lemma}\label{lem:4.1}
Let $w_n$ be a  solution of  \eqref{eq:3.3a}.
 There is a constant, independent of $n$, such that

\[
\frac{1}{r^{N+1-2\alpha}}\int_{\partial \mathcal{ B}_{r}^+(z)\cap\{y>0\}}y^{1-2\alpha}w_n\,dS\leq C
\]
for all $r\ge \bar C\sigma_n^{-1/2}$  and  $z=(z',0)$ with $z'\in \Omega$.

\end{Lemma}

\begin{proof} For $X = (x,y), z=(z',0)$,
\[
\Gamma(X,z) = \frac 1{|X-z|^{N-2\alpha}} - \frac 1{s^{N-2\alpha}}
\]
satisfies
\[
div(y^{1-2\alpha}\nabla_X \Gamma(X,z)) = 0 \quad {\rm in}\quad  \mathcal{B}_s(z)\setminus \{z\}; \quad \Gamma
(X,z)= 0, \quad X\in \partial \mathcal{B}_s(z),
\]
where $\mathcal{B}_s(z)\subset\mathbb{R}^{N+1}$ is a ball centered at $z$ with radius $s$.

Denote $f_n = 2|v_n|^{2_\alpha^*-1} +A$. Integrating by parts,  we find that for  $\delta\in (0, s)$,

\begin{equation}\label{1-4-8}
\begin{split}
0=&\int_{\mathcal {B}_s^+(z)\setminus \mathcal {B}_\delta^+(z)}div(y^{1-2\alpha}\nabla w_n)\Gamma(X,z)\,dX\\
= &\int_{\partial( \mathcal {B}_s^+(z)\setminus \mathcal {B}_\delta^+(z))}y^{1-2\alpha}\frac{\partial w_n}{\partial n}\Gamma(X,z)\,dS - \int_{
\partial( \mathcal {B}_s^+(z)\setminus \mathcal {B}_\delta^+(z))}y^{1-2\alpha} w_n \frac{\partial \Gamma}{\partial n}\,dS\\
=&\int_{\{y=0\}\cap (\mathcal {B}_s(x)\setminus \mathcal {B}_\delta(z))}f_n \Gamma(X,z)\,dX+
\int_{\{y>0\}\cap \partial \mathcal {B}_\delta(z)}y^{1-2\alpha}\frac{\partial w_n}{\partial n}\Gamma(X,z)\,dS
\\
&- \int_{\{y>0\}\cap \partial( \mathcal {B}_s(z)\setminus \mathcal {B}_\delta(z))
}y^{1-2\alpha} w_n \frac{\partial \Gamma}{\partial n}\,dS,
\end{split}
\end{equation}
since

\begin{equation}\label{2-4-8}
\Gamma(X,z)= 0, \quad X\in \partial \mathcal {B}_s(z),
\end{equation}
and

\[
y^{1-2\alpha}  \frac{\partial \Gamma(X,z)}{\partial n}=-\frac{(N-2\alpha) y^{2-2\alpha}}
{
|X-z|^{N-2\alpha +2}}=0,\quad X\in \{ y=0\} \cap ( \mathcal {B}_s(z)\setminus \mathcal {B}_\delta(z)).
\]
Differentiating \eqref{1-4-8} with respect to  $s$, using \eqref{2-4-8},
we are led to

\begin{equation}\label{3-4-8}
\begin{split}
&
\int_{\{y=0\}\cap (\mathcal {B}_s(x)\setminus \mathcal {B}_\delta(z))}f_n \frac{N-2\alpha}{s^{N-2\alpha+1}}
\,dX+
\int_{\{y>0\}\cap \partial \mathcal {B}_\delta(z)}y^{1-2\alpha}\frac{\partial w_n}{\partial n}
\frac{N-2\alpha}{s^{N-2\alpha+1}}
\,dS
\\
&+ \frac{d}{ds}\int_{\{y>0\}\cap \partial \mathcal {B}_s(z)
}y^{1-2\alpha} w_n \frac{N-2\alpha}{s^{N-2\alpha+1}}\,dS=0.
\end{split}
\end{equation}
Letting $\delta\to 0$ in \eqref{3-4-8}, we obtain the following formula

\begin{equation}\label{4-4-8}
\frac{1}{s^{N-2\alpha+1}}\int_{\{y=0\}\cap \mathcal {B}_s(x)}f_n
\,dX+\frac{d}{ds}\Bigl(\frac{1}{s^{N-2\alpha+1}}\int_{\{y>0\}\cap \partial \mathcal {B}_s(z)
}y^{1-2\alpha} w_n \,dS\Bigr)=0,
\end{equation}
since

\[
y^{1-2\alpha}\frac{\partial w_n}{\partial n}\to  2|v_n(x,0)|^{2^*_\alpha-1} +A,\quad
\text{as}\; y\to 0.
\]
From
\[
\begin{split}
&\int_{\mathcal {B}_s(z)\cap\{y>0\}}y^{1-2\alpha}w_n\,dX \\
&\leq\bigg(\int_{\mathcal {B}_s(z)\cap\{y>0\}}y^{1-2\alpha}\,dX\bigg)^{\frac 12}\bigg(\int_{
\mathcal{C}}y^{1-2\alpha}w_n^2\,dX\bigg)^{\frac 12}
\le C,
\end{split}
\]
we can find a $r_n\in
\bigl[\frac12,1\bigr]$, such that

\[
\frac{1}{r_n^{N+1-2\alpha}}\int_{\partial \mathcal {B}_{r_n}(z)\cap\{y>0\}}y^{1-2\alpha}w_n\,dS\le C.
\]

Integrating \eqref{4-4-8}  from $r$  to $r_n$, we obtain

\begin{equation}\label{eq:4.3}
\begin{split}
&\frac{1}{r^{N+1-2\alpha}}\int_{\partial \mathcal {B}_{r}(z)\cap\{y>0\}}y^{1-2\alpha}w_n\,dS\\
=&\frac{1}{r_n^{N+1-2\alpha}}\int_{\partial \mathcal {B}_{r_n}(y)\cap\{y>0\}}y^{1-2\alpha}w_n\,dS
+\int_{r}^{r_n}
\frac{1}{t^{N+1-2\alpha}}\int_{ \mathcal {B}_t(z)\cap\{y=0\}
} f_n\,dSdt\\
\le & C+\int_{r}^{r_n}\frac{1}{t^{N+1-2\alpha}}\int_{\{y=0\}\cap \mathcal {B}_t(z)}\bigl(2|v_n|^{2^*_\alpha-1} + A\bigr)\,dxdt\\
\le & C+C\int_{r}^{r_n}\frac{1}{t^{N+1-2\alpha}}\int_{\{y=0\}\cap \mathcal {B}_t(y)}\bigl(
w_n^{2^*_\alpha-1}+ A\bigr)\,dxdt,\\
\end{split}
\end{equation}
since $|v_n|\le w_n$.

It is easy to check

\begin{equation}\label{eq:4.4}
\int_{r}^{r_n}\frac{1}{t^{N+1-2\alpha}}\int_{\{y=0\}\cap \mathcal {B}_t(z)} A\,dxdt\leq C\int_{r}^{r_n}t^{2\alpha-1}\,dt\leq C.
\end{equation}

By Proposition~\ref{prop:3.1}, we know that
$\|w_n(\cdot,0)\|_{q_1,q_2,\sigma_n}\leq C$ for any $\frac{N}{N-2\alpha}<q_2< 2_\alpha^*<q_1$.
Let  $q_1>2_\alpha^*$ large  such that

\[
-\frac{(N+2\alpha)}{q_1(N-2\alpha)}+2\alpha-1>-1.
\]

Let
\[
 q_2=\frac{N+2\alpha}{N-2\alpha}.
 \]
Then, we can choose $ v_{1,n}$, and $
v_{2,n}$, such that $|w_n(x,0)|\le v_{1,n}+ v_{2,n}$, and
\[
\| v_{1,n}\|_{q_1}\le C,
\]
and
\[
\| v_{2,n}\|_{q_2}\le C\sigma_n^{\frac{N}{2^*_\alpha}-\frac{N}{q_2}}.
\]

We have

\begin{equation}\label{eq:4.5}
\begin{split}
&\int_{r}^{r_n}\frac{1}{t^{N+1-2\alpha}}\int_{\{y=0\}\cap \mathcal {B}_t(z)}|v_{1,n}|^{2^*_\alpha-1}\,dxdt\\
\le  &\int_{r}^{r_n}\frac1{t^{N+1-2\alpha}}
\Bigl(\int_{\mathcal {B}_t(z)\cap \{y=0\}}|v_{1,n}|^{q_1}\,dx\Bigr)^{\frac{N+2\alpha}{(N-2\alpha)q_1}}t^{N(1-\frac{N+2\alpha}{(N-2\alpha)q_1})}\,dt\\
\le & C\int_{r}^{1} t^{-\frac{(N+2\alpha)}{q_1(N-2\alpha)}+2\alpha-1 }\,dt\le C.
\end{split}
\end{equation}

On the other hand, noting that  $r\ge \bar C\sigma_n^{-1/2}$,
\begin{equation}\label{neq:4.6}
\begin{split}
&\int_{r}^{r_n}\frac{1}{t^{N+1-2\alpha}}\int_{\{y=0\}\cap \mathcal {B}_t(z)}|v_{2,n}|^{2^*_\alpha-1}\,dxdt\\
\le  &C \sigma_n^{(\frac N{2^*_\alpha}-\frac N{q_2})q_2}\int_{r}^{r_n}\frac1{t^{N+1-2\alpha}}\,dt\le C \sigma_n^{(\frac N{2^*_\alpha}-\frac N{q_2})q_2} r^{2\alpha -N}\\
\le & C \sigma_n^{(\frac N{2^*_\alpha}-\frac N{q_2})q_2+\frac{N-2\alpha}2}= C.
\end{split}
\end{equation}

 Combining \eqref{eq:4.4}-\eqref{neq:4.6}, we obtain
\begin{equation}\label{eq:4.6}
\int_{r}^{r_n}\frac{1}{t^{N+1-2\alpha}}\int_{\{y=0\}\cap \mathcal {B}_t(z)}\bigl(2|w_n|^{2^*_\alpha-1} + A\bigr)\,dxdt\leq C,
\end{equation}
and then
\[
\frac{1}{r^{N+1-2\alpha}}\int_{\partial \mathcal {B}_{r}^+(z)\cap\{y>0\}}y^{1-2\alpha}w_n\,dS\leq C.
\]

\end{proof}

Let us recall the Muckenhoupt class $A_p$ for $p>1$:

\[
A_p= \bigl\{ w:\; \sup_{\mathcal {B}} \Bigl(\frac1{|\mathcal {B}|}\int_{
\mathcal {B}} |w|\Bigr) \Bigl(\frac1{|\mathcal {B}|}
\int_{\mathcal {B}}|w|^{-\frac 1{p-1}}
\Bigr)^{p-1}\le C,\;\; \text{for all ball  $\mathcal {B}$ in $\mathbb R^{N+1}$}\bigr\}.
\]
It is easy to check that $y^{1-2\alpha}\in  A_2$.

Denote $\|u\|_{L^p(E,y^{1-2\alpha})} = (\int_E y^{1-2\alpha}|u|^p\,dx)^{\frac 1p}$.
We have the following result  \cite{FKS}:

\begin{Lemma}\label{lm:4.2}
Let $ \mathcal D $ be an open bounded set in $\mathbb{R}^{N+1}$.
 There exist constants  $\delta>0$  and
$C>0$ depending  only on $N$ and $\mathcal D$,    such that for all $u\in C_0^\infty(\mathcal D)$ and all $k$ satisfying $1\leq k\leq \frac N{N-1}+\delta$,
\begin{equation}\label{eq:4.1}
\|u\|_{L^{2k}(\mathcal D, y^{1-2\alpha})}\leq C\|\nabla u\|_{L^{2}(\mathcal D, y^{1-2\alpha})}.
\end{equation}

\end{Lemma}

Let $D^*$ be an open set in $\mathbb{R}^N$.
Consider the following problem:

\begin{equation}\label{1-5-8}
\begin{cases}
div(y^{1-2\alpha} \nabla w)=0, & (x,y)\in \mathcal{C}_{D^*};\\
-y^{1-2\alpha}\frac{\partial w}{\partial y} = a(x) w, & x\in D^*,\; y=0,
\end{cases}
\end{equation}
where  $a(x)\geq 0$ and $a\in L^\infty_{loc}(\mathbb{R}^N)$.  We have
the following estimate:

\begin{Lemma}\label{lem:4.3}  Suppose that
$w$ is a solution of \eqref{1-5-8}.   If there is a small constant $\delta>0$ such that
\[
\int_{\mathcal {B}_1(z)\cap\{y=0\}}|a|^{\frac N{2\alpha}}\,dx\leq \delta,
\]
for any  $\mathcal {B}_1(z)\cap \{ y=0\}\subset
 D^*$, $z= (x,0)$, then for any $p\geq 1$, there is a constant $C = C(p)>0$ such that
\begin{equation}\label{eq:4.7}
\|w\|_{L^{p}(\mathcal {B}^+_{ 1/2}(z),y^{1-2\alpha})}\leq C\|w\|_{L^{1}(\mathcal {B}^+_{1}(z),y^{1-2\alpha})},
\end{equation}
and
\begin{equation}\label{eq:4.7a}
\bigg(\int_{\mathcal {B}^+_{r}(z)\cap \{y=0\}}w^{p}\,dx\bigg)^{\frac 1p}
\leq \frac C{(R-r)^{\frac\sigma\kappa}}\|w\|_{L^1(\mathcal {B}^+_{R}(z),y^{1-2\alpha})}
\end{equation}
for $p\geq 1,\, 0<\sigma\leq 1$ and $0<\kappa<1$.
\end{Lemma}

\begin{proof}
We only need to prove the result for $p>2^*_\alpha$.
Let $1\geq R>r>0$. Define $\xi\in C_0^2(\mathcal B_R(z))$, with $\xi=1$ in $\mathcal B_r(z)$, $0\leq
\xi\leq 1$, and $|\nabla\xi|\leq \frac 2{R-r}$. Let $q=\frac p{2^*_\alpha}$ and $R$ be small so that
$\varphi =\xi^2 w^{2q-1}\in H_{0,L}^1(\mathcal{C}_{D^*})$.  We have
\[
\int_{\mathcal{C}_{D^*}}y^{1-2\alpha}\nabla w \nabla\varphi\,dxdy=
 \int_{D^* \cap \{y=0\}} aw\varphi\,dx,
\]
and
\[
\begin{split}
&\int_{\mathcal{C}_{D^*}}y^{1-2\alpha}\nabla w \nabla\varphi\,dxdy\\
&\geq \frac{2q-1}{2q^2}\int_{\mathcal{C}_{D^*}}y^{1-2\alpha}|\nabla(\xi w^q)|^2\,dxdy - \frac C{(R-r)^2}\int_{\mathcal B^+_{R}(z)}y^{1-2\alpha}w^{2q}\,dx.\\
\end{split}
\]
Hence,
\begin{equation}\label{eq:4.8}
\begin{split}
&\int_{\mathcal{C}_{D^*}}y^{1-2\alpha}|\nabla(\xi w^q)|^2\,dxdy\\
\leq &\frac C{(R-r)^2}\int_{\mathcal B^+_{R}(z)}y^{1-2\alpha}w^{2q}\,dx + \int_{
D^*\times\{0\}} aw\varphi\,dx\\
\leq &\frac C{(R-r)^2}\int_{\mathcal B^+_{R}(z)}y^{1-2\alpha}w^{2q}\,dx \\
&+\bigg(\int_{ \mathcal B^+_{R}(z)\cap \{y=0\}}|a|^{\frac N{2\alpha}}\,dx\bigg)^{\frac {2\alpha}N}\bigg(\int_{D^* \times\{0\}}(\xi w^q)^{2^*_\alpha}\,dx\bigg)^{\frac{N-2\alpha}N}\\
\leq &\frac C{(R-r)^2}\int_{\mathcal B^+_{R}(z)}y^{1-2\alpha}w^{2q}\,dx +\delta^{\frac {2\alpha}N}\bigg(\int_{D^* \times\{0\}}(\xi w^q)^{2^*_\alpha }\,dx\bigg)^{\frac{N-2\alpha}N}.\\
\end{split}
\end{equation}

By the trace inequality, we obtain

\begin{equation}\label{eq:4.9}
\begin{split}
&\int_{\mathcal{C}_{D^*}}y^{1-2\alpha}|\nabla(\xi w^q)|^2\,dxdy
\\
\leq &\frac C{(R-r)^2}\int_{\mathcal B^+_{R}(z)}y^{1-2\alpha}w^{2q}\,dxdy
+C\delta^{\frac {2\alpha}N}
\int_{\mathcal{C}_{D^*}}y^{1-2\alpha}|\nabla(\xi w^q)|^2\,dxdy.
\end{split}
\end{equation}
So,  if $\delta>0$ is small, we obtain

\begin{equation}\label{eq:4.10}
\int_{\mathcal{C}_{D^*}}y^{1-2\alpha}|\nabla(\xi w^q)|^2\,dxdy
\leq \frac C{(R-r)^2}\int_{\mathcal B^+_{R}(z)}y^{1-2\alpha}w^{2q}\,dxdy
\end{equation}
for $0<r<R<1$. By Lemma \ref{lm:4.2},

\begin{equation}\label{10-5-8}
\bigg(\int_{\mathcal B^+_{R}(z)}y^{1-2\alpha}(\xi w^{q})^{2t}\,dx\bigg)^{\frac1t}\le
C\int_{\mathcal{B}^+_{R}(z)}y^{1-2\alpha}|\nabla(\xi w^q)|^2\,dxdy
\end{equation}
for some $t>1$. As a result,  we obtain from \eqref{eq:4.10}  and \eqref{10-5-8},

\begin{equation}\label{eq:4.11}
\bigg(\int_{\mathcal B^+_{r}(z)}y^{1-2\alpha}w^{2tq}\,dx\bigg)^{\frac1t}
\leq \frac C{(R-r)^2}\int_{\mathcal B^+_{R}(z)}y^{1-2\alpha}w^{2q}\,dxdy,
\end{equation}
which yields

\begin{equation}\label{eq:4.12}
\bigg(\int_{\mathcal B^+_{r}(z)}y^{1-2\alpha}w^{2tq}\,dx\bigg)^{\frac1{2tq}}
\leq \frac C{(R-r)^{\frac 1q}}\bigg(\int_{\mathcal B^+_{R}(z)}y^{1-2\alpha}w^{2q}\,dxdy\bigg)^{\frac 1{2q}}
\end{equation}
for $0<r<R<1$. Note that if $p>q\geq 1$, by H\"{o}lder's inequality,

\begin{equation}\label{eq:4.13}
\bigg(\int_{\mathcal B^+_{R}(z)}y^{1-2\alpha}w^{q}\,dx\bigg)^{\frac1{q}}
\leq C\bigg(\int_{\mathcal B^+_{R}(z)}y^{1-2\alpha}w^{p}\,dxdy\bigg)^{\frac 1{p}}.
\end{equation}

Using \eqref{eq:4.13} and iterating \eqref{eq:4.12} we obtain that there is $\sigma>0$ such that
\begin{equation}\label{eq:4.14}
\bigg(\int_{\mathcal B^+_{r}(z)}y^{1-2\alpha}w^{p}\,dx\bigg)^{\frac1p}
\leq \frac C{(R-r)^\sigma}\bigg(\int_{\mathcal B^+_{R}(z)}y^{1-2\alpha}w^{2^*_\alpha}\,dxdy\bigg)^{\frac 1{2^*_\alpha}}
\end{equation}
for $p>2^*_\alpha$ and $0<r<R<1$. By H\"{o}lder's inequality,

\[
\bigg(\int_{\mathcal B^+_{R}(z)}y^{1-2\alpha}w^{2^*_\alpha}\,dxdy\bigg)^{\frac 1{2^*_\alpha}}\leq \|w\|^{\kappa}_{L^1(\mathcal B^+_{R}(z),y^{1-2\alpha})}\|w\|^{1-\kappa}_{L^p(\mathcal B^+_{R}(z),y^{1-2\alpha})}.
\]
Hence,
\[
\|w\|_{L^p(\mathcal B^+_{r}(z),y^{1-2\alpha})}\leq \frac 12\|w\|_{L^p(\mathcal B^+_{R}(z),y^{1-2\alpha})} + \frac C{(R-r)^{\frac\sigma\kappa}}\|w\|_{L^1(\mathcal B^+_{R}(z),y^{1-2\alpha})}.
\]
By iteration, we obtain

\begin{equation}\label{20-5-8}
\|w\|_{L^p(\mathcal B^+_{r}(z),y^{1-2\alpha})}\leq  \frac C{(R-r)^{\frac\sigma\kappa}}\|w\|_{L^1(\mathcal B^+_{R}(z),y^{1-2\alpha})}
\end{equation}
for $p>2^*_\alpha$ and $0<r<R<1$.

Finally, \eqref{eq:4.10}, \eqref{20-5-8} and the trace inequality  imply that
\begin{equation}\label{eq:4.15}
\bigg(\int_{\mathcal B^+_{r}(z)\cap \{y=0\}}w^{p}\,dx\bigg)^{\frac 1p}
\leq \frac C{(R-r)^{\frac\sigma\kappa}}\|w\|_{L^1(\mathcal B^+_{R}(z),y^{1-2\alpha})}.
\end{equation}

\end{proof}

\begin{proof}[Proof of Proposition \ref{prop:4.1}]

It follows from Lemma~\ref{lem:4.1}
\[
\frac{1}{r^{N+1-2\alpha}}\int_{\partial
\mathcal  B^+_{r}((x_n,0))}y^{1-2\alpha}w_n\,dS\leq C,
\]
which gives

\[
\int_{\mathcal A_n^2}y^{1-2\alpha}w_n\,dX\leq C\int_{(C+1)\sigma_n^{
-\frac12}}^{(C+4)\sigma_n^{
-\frac12}}
r^{N+1-2\alpha}\,dr\leq C\sigma_n^{-\frac {N+2-2\alpha}2}.
\]
In particular

\begin{equation}\label{10-1-9}
\int_{\mathcal  B^+_{\sigma_n^{-\frac 12}}(z)}y^{1-2\alpha}w_n\,dX\leq C\sigma_n^{-\frac {N+2-2\alpha}2},\quad \forall\; z\in \mathcal {A}_n^2.
\end{equation}

Let
\[
\tilde v_n(X)= v_n(\sigma_n^{-\frac12} X),\quad X=(x,y)\in\mathcal{C}_{\Omega_n},
\]
where $\Omega_n=\{x: \sigma_n^{-\frac12} x\in\Omega\}$. Then $\tilde v_n$ satisfies
\[
\begin{cases}
div(y^{1-2\alpha}\nabla \tilde v_n) = 0,& \text{in}\; \mathcal{C}_{\Omega_n},\\
\tilde v_n=0,& \text{on}\; \partial_L\mathcal{C}_{\Omega_n},\\
y^{1-2\alpha}\frac{\partial\tilde v_n}{\partial \nu} = \sigma_n^{-\alpha}(|\tilde v_n(x,0)|^{p_n-2}\tilde v_n(x,0) + \lambda\tilde v_n(x,0)), & \text{on}\; \Omega_n\times\{0\}.
\end{cases}
\]

Let $\xi=\sigma_n^{\frac 12} z$.  Since $B^+_{\sigma_n^{-\frac12}}(z)$, $z\in
\mathcal{A}_n^2$,  does not contain any concentration point of $v_n$, we can
deduce

\[
\begin{split}
&\int_{\mathcal B_1(\xi)\cap\{y=0\}}|\sigma_n^{-\alpha}\bigl(
|\tilde v_n(x,0)|^{p_n-2}+\lambda\bigr)|^{\frac N{2\alpha}}\,dx\\
\leq &C\int_{\mathcal B_1(\xi)\cap\{y=0\}}|\sigma_n^{-\alpha}\bigl( |\tilde v_n(x,0)|^{2_\alpha^*-2}+1\bigr)|^{\frac N{2\alpha}}\,dx\\
\leq  & C \int_{\mathcal B_{
\sigma_n^{-\frac 12}}(z)\cap\{y=0\}}|v_n|^{2^*_\alpha}\,dx+C\sigma_n^{-\frac N{2}}\to 0\\
\end{split}
\]
as $n\to\infty$. Thus, by Lemmas \ref{lem:4.1}  and \ref{lem:4.3}, noting $|v_n|\leq w_n$, we
obtain
\[
\begin{split}
&\|\tilde v_n\|_{L^p(\mathcal B^+_{\frac 12}(\xi),y^{1-2\alpha})}
\leq C\int_{\mathcal B^+_1(\xi)}y^{1-2\alpha}|\tilde v_n|\,dxdy\\
\leq & C{\sigma_n^{\frac {N+2-2\alpha}2}} \int_{\mathcal B^+_{ \sigma_n^{-\frac12}}(z)}y^{1-2\alpha}|w_n|\,dxdy\\
\leq&  C{\sigma_n^{\frac {N+2-2\alpha}2}}\int_{\mathcal{A}_n^2}y^{1-2\alpha}|w_n|\,dxdy\\
\leq &  C{\sigma_n^{\frac {N+2-2\alpha}2}} {\sigma_n^{-\frac {N+2-2\alpha}2}}\leq C.
\end{split}
\]
By \eqref{eq:4.15}, we also have
\[
\bigg(\int_{\mathcal B^+_{\frac 12}(\xi)\cap \{y=0\}}|\tilde v_n|^{p}\,dx\bigg)^{\frac 1p}
\leq C\|\tilde v_n\|_{L^1(\mathcal B^+_{1}(\xi),y^{1-2\alpha})}\leq C.
\]
As a result,
\[
\sigma_n^{\frac{N+2-2\alpha}{2p}} \Bigl(\int_{\mathcal B_{\frac 12
\sigma_n^{-\frac12}}(z)}y^{1-2\alpha}|v_n|^{p}\,dxdy\Bigr)^{\frac 1p}\leq C,\quad \forall\; z\in \mathcal{A}_n^2.
\]
Thus,
\[
\int_{\mathcal{A}_n^2}y^{1-2\alpha}|v_n|^p\leq C\sigma_n^{-\frac{N+2-2\alpha}{2}}.
\]
Similarly,
\[
\int_{\mathcal{A}_n^2\cap\{y=0\}}|v_n|^p\leq C\sigma_n^{-\frac{N}{2}}.
\]

\end{proof}

\begin{Proposition}\label{prop:4.2}
We have

\begin{equation}\label{eq:4.19}
\begin{split}
&\int_{\mathcal{A}_n^3}y^{1-2\alpha}|\nabla v_n|^2\,dxdy\\
\le & C\sigma_n\int_{\mathcal{A}_n^2}y^{1-2\alpha}|w_n|^{2}\,dxdy+
C\int_{\mathcal{A}_n^2\times\{y=0\}}|w_n|^{2^*_\alpha}\,dx +C\int_{\mathcal{A}_n^2\times\{y=0\}}|w_n|^{2}\,dx.
\end{split}
\end{equation}
In particular,

\begin{equation}\label{2ln1}
\int_{\mathcal{A}_n^3}y^{1-2\alpha}|\nabla v_n|^2\,dxdy\le C\sigma_n^{-\frac{N-2\alpha}2}.
\end{equation}

\end{Proposition}

\begin{proof}

Let $\varphi_n\in C_0^2(\mathcal{A}_n^2)$ be a function with $\varphi_n=1$ in
$\mathcal{A}_n^3$; $0\le \varphi_n\le 1$ and $|\nabla\varphi_n|\le C\sigma_n^{\frac12}$.
From
\[
\begin{split}
&\int_{\mathcal{C}_\Omega}y^{1-2\alpha}\nabla v_n \nabla(\varphi^2_n v_n)\,dxdy=
\int_{\Omega\times\{y=0\}}\bigl(
|v_n|^{2^*_\alpha-2}+\lambda\bigr) v_n\varphi_n^2 v_n\,dx
\\
\le & C\int_{\Omega\times\{y=0\}}\bigl(
|w_n|^{2^*_\alpha}+ w_n^2\bigr)\varphi_n^2 \,dx,
\end{split}
\]
we can prove \eqref{eq:4.19}.

On the other hand, it follows from Proposition~\ref{prop:4.1} that
\begin{equation}\label{1-9-2}
\begin{split}
&\sigma_n\int_{\mathcal{A}_n^2}y^{1-2\alpha}|w_n|^{2}\,dxdy+
\int_{\mathcal{A}_n^2\times\{y=0\}}|w_n|^{2^*_\alpha}\,dx +
\int_{\mathcal{A}_n^2\times\{y=0\}}|w_n|^{2}\,dx\\
\le &C\sigma_n\sigma_n^{-\frac {N+2-2\alpha}2}+  C\sigma_n^{-\frac{N}{2}}
\le C\sigma_n^{-\frac{N-2\alpha}2}.
\end{split}
\end{equation}
It yields from \eqref{eq:4.19} that
\[
\int_{\mathcal A_n^3}y^{1-2\alpha}|\nabla v_n|^2\,dxdy\le C\sigma_n^{-\frac{N-2
\alpha}2}.
\]

\end{proof}

\section {Existence of infinitely many bound state solutions}

\setcounter{equation}{0}

Firstly, we have the following local Pohozaev identity.
\begin{Lemma}\label{lem:5.1}
Let $v$ be a solution of (\ref{eq:1.3}). Then for any smooth subset $
\mathcal M\subset \mathcal{C}_\Omega$, $v$ satisfies
\begin{equation}\label{eq:5.1}
\begin{split}
&\frac {N-2\alpha}2\int_{\partial \mathcal M}y^{1-2\alpha}v\frac{\partial v}{\partial\nu}\,dS\\
&=\frac 12 \int_{\partial \mathcal M}y^{1-2\alpha}|\nabla v|^2(X-z_0,\nu)\,dS - \int_{\partial \mathcal M}y^{1-2\alpha}(\nabla v, X-z_0)\frac{\partial v}{\partial\nu}\,dS,\\
\end{split}
\end{equation}
where $\nu $ is the outward normal to $\partial S$,  and  $z_0
\in \mathbb {R}^{N+1}$.
\end{Lemma}

{\bf Proof of Theorem \ref{thm:1.1}} We argue by contradiction. Suppose the assertion is not true.
Choose $t_n\in [\bar C+2,\bar C + 3]$ so that
\begin{equation}\label{eq:5.2}
\begin{split}
&\int_{\bigl(\partial \mathcal B_{t_n\sigma_n^{-\frac 12}}((x_n,0))\bigr)\cap \mathcal C_\Omega}
y^{1-2\alpha}\big(\sigma_n^{-\frac 12} |\nabla v_n|^2 +
\sigma_n^{\frac12}v_n^2\big)\,dS\\
&+\sigma_n^{\alpha-\frac 12}\int_{\bigl(\partial \mathcal B_{t_n\sigma_n^{-\frac 12}}((x_n,0))\bigr)\cap ( \Omega
\times \{0\})}
\bigl( | v_n|^{2^*_\alpha} +  v_n^2\bigr)\,dS
\\
&\leq \int_{\mathcal{A}_n^3}y^{1-2\alpha}\big(|\nabla v_n|^2 + \sigma_n v_n^2
\big)\,dxdy+\sigma_n^\alpha
\int_{\mathcal{A}_n^3\cap\{y=0\}}\bigl( | v_n|^{2^*_\alpha} +  v_n^2\bigr)\,dx,\\
\end{split}
\end{equation}
By Propositions \ref{prop:4.1} and \ref{prop:4.2},
\begin{equation}\label{eq:5.3}
\begin{split}
&\int_{\bigl(\partial \mathcal B_{t_n\sigma_n^{-\frac 12}}((x_n,0))\bigr)\cap \mathcal C_\Omega}
y^{1-2\alpha}\big(\sigma_n^{-\frac 12} |\nabla v_n|^2 +
\sigma_n^{\frac12}v_n^2\big)\,dS\\
&+\sigma_n^{\alpha-\frac 12}\int_{\bigl(\partial \mathcal B_{t_n\sigma_n^{-\frac 12}}((x_n,0))\bigr)\cap ( \Omega
\times \{0\})}
\bigl( | v_n|^{2^*_\alpha} +  v_n^2\bigr)\,dS
\\
&\leq
C\sigma_n^{-\frac{N-2\alpha}2} + C\sigma_n^\alpha
\sigma_n^{-\frac{N}2}=C'\sigma_n^{-\frac{N-2\alpha}2}.
\end{split}
\end{equation}

Let $p_n = 2_\alpha^* - \varepsilon_n$. Applying Lemma \ref{lem:5.1} to $v_n$ on $
\mathcal B_n = \mathcal B_{t_n\sigma_n^{-\frac 12}}((x_n,0))\cap \mathcal C_\Omega
\subset \mathbb{R}^{N+1}$  and  $z_0=(x_0,0)$, we obtain
\begin{equation}\label{eq:5.4}
\begin{split}
&\frac {N-2\alpha}2\int_{\partial B_n}y^{1-2\alpha}v_n\frac{\partial v_n}{\partial\nu
}\,dS\\
&=\frac 12 \int_{\partial B_n}y^{1-2\alpha}|\nabla v_n|^2(X-z_0,\nu)\,dS - \int_{\partial B_n}y^{1-2\alpha}(\nabla v_n, X-z_0)\frac{\partial v_n}{\partial\nu}\,dS.\\
\end{split}
\end{equation}

From the fact that

\[
y^{1-2\alpha}v_n\frac{\partial v_n}{\partial\nu}= |v_n(x,0)|^{p_n-2} v_n(x,0)+
\lambda v_n(x,0),\quad \text{on}\; y=0,
\]
 we obtain from \eqref{eq:5.4}
\begin{equation}\label{eq:5.5}
\begin{split}
&\frac {N-2\alpha}2\int_{\mathcal B_n\cap\{y=0\}}( |v_n|^{p_n}
+\lambda v_n^2)\,dx + \frac {N-2\alpha}2\int_{\partial \mathcal B_n\cap\{y>0\}}y^{1-2\alpha}v_n\frac{\partial v_n}{\partial\nu}\,dS\\
&= \frac 12 \int_{\mathcal B_n\cap\{y=0\}}y^{1-2\alpha}|\nabla v_n|^2(x-x_0,\nu)\,dx + \frac 12 \int_{\partial \mathcal B_n\cap\{y>0\}}y^{1-2\alpha}|\nabla v_n|^2(X-z_0,\nu)\,dS\\
&- \int_{\mathcal B_n\cap\{y=0\}}y^{1-2\alpha}(\nabla v_n, x-x_0)\frac{\partial v_n}{\partial\nu}\,dx - \int_{\partial \mathcal B_n\cap\{ y>0\}}y^{1-2\alpha}(\nabla v_n, X-z_0)\frac{\partial v_n}{\partial\nu}\,dS.\\
\end{split}
\end{equation}

Noting  that $x-x_0\perp \nu$ on $\mathcal B_n\cap\{y=0\}$,  we find
\[
\int_{\mathcal B_n\cap\{y=0\}}y^{1-2\alpha}|\nabla v_n|^2(x-x_0,\nu)\,dx = 0.
\]

On the other hand,
\begin{equation}\label{eq:5.6}
\begin{split}
&- \int_{\mathcal B_n\cap\{y=0\}}y^{1-2\alpha}(\nabla v_n, x-x_0)\frac{\partial v_n}{\partial\nu}\,dx\\
&= - \int_{\mathcal B_n\cap\{y=0\}}(\nabla_x v_n, x-x_0)(|v_n|^{p_n-2} v_n+ \lambda v_n)\,dx.\\
&= - \int_{\mathcal B_n\cap\{y=0\}}(\nabla_x(\frac 1{p_n} |v_n|^{p_n} +\frac12\lambda v_n^2), x-x_0)\,dx\\
&= N\int_{\mathcal B_n\cap\{y=0\}}(\frac 1{p_n} |v_n|^{p_n} +\frac12\lambda v_n^2)\,dx - \int_{\partial (\mathcal B_n\cap\{y=0\})}(\frac 1{p_n} |v_n|^{p_n} +\frac12\lambda v_n^2)\langle x-x_0,\nu_x\rangle\,dS.\\
\end{split}
\end{equation}
So equation \eqref{eq:5.5} becomes

\begin{equation}\label{eq:5.7}
\begin{split}
&\bigg(\frac N{p_n} - \frac {N-2\alpha}2\bigg)\int_{\mathcal B_n\cap\{y=0\}}|v_n|^{p_n}\,dx + \bigg(\frac N{2} - \frac {N-2\alpha}2 \bigg)
\lambda\int_{\mathcal B_n\cap\{y=0\}} v_n^{2}\,dx \\
&= \int_{\partial (\mathcal B_n\cap\{y=0\})}(\frac 1{p_n} |v_n|^{p_n} +\frac12\lambda v_n^2)\langle x-x_0,\nu_x\rangle\,dS\\
&+ \frac {N-2\alpha}2\int_{\partial \mathcal B_n\cap\{y>0\}}y^{1-2\alpha}v_n\frac{\partial v_n}{\partial\nu}\,dS\\
& -\frac 12 \int_{\partial \mathcal B_n\cap\{y>0\}}y^{1-2\alpha}|\nabla v_n|^2(X-z_0,\nu)\,dS\\
&+\int_{\partial \mathcal B_n\cap\{y>0\}}y^{1-2\alpha}(\nabla v_n, X-z_0)\frac{\partial v_n}{\partial\nu}\,dS.\\
\end{split}
\end{equation}

We decompose

\[
\partial\mathcal  B_n \cap \{y>0\}=  \partial_i
\mathcal B_n\cup \partial_e \mathcal B_n,
\]
where $\partial_i \mathcal B_n = \partial\mathcal  B_n \cap\mathcal C_\Omega$ and
$\partial_e \mathcal B_n = \mathcal  B_n \cap\partial_L \mathcal C_\Omega$.

Now, we have two cases:

\begin{itemize}
\item [(i)] $\mathcal B_{t_n\sigma_n^{-\frac 12}}((x_n,0))\cap\{y>0\}\cap \big(\mathbb{R}^{N+1}\setminus \mathcal C_\Omega\big)\not=\emptyset$,

\item [(ii)] $\mathcal B_{t_n\sigma_n^{-\frac 12}}((x_n,0))\cap\{y>0\}\subset \mathcal C_\Omega$.

\end{itemize}

In case $(i)$, we take $x_0\in \mathbb{R}^{N}\setminus \Omega$ with $|x_n - x_0|\leq 2t_n\sigma_n^{-\frac 12}$, $\nu \cdot(X-(x_0,0))\leq 0$ on
$\partial_e \mathcal  B_n$, where $\nu$ is the outward normal to $\partial_L
\mathcal C_\Omega$.  Since  $v_n=0 $ on $\partial_L
\mathcal C_\Omega$, we find

\begin{equation}\label{20-1-9}
\begin{split}
& -\frac 12 \int_{\partial_e \mathcal B_n}y^{1-2\alpha}|\nabla v_n|^2(X-z_0,\nu)\,dS\\
&+\int_{\partial_e \mathcal B_n}y^{1-2\alpha}(\nabla v_n, X-z_0)\frac{\partial v_n}{\partial\nu}\,dS.\\
=& \frac 12 \int_{\partial_e \mathcal B_n}y^{1-2\alpha}|\nabla v_n|^2(X-z_0,\nu)\,dS\le
0.
\end{split}
\end{equation}

In case $(ii)$,  $ \partial_e \mathcal B_n=\emptyset$.    We choose $x_0=x_n$.

Noting that  $ p_n\le 2^*_\alpha$  and
 $v_n=0 $ on $\partial_L
\mathcal C_\Omega$, we obtain from \eqref{eq:5.7}

\begin{equation}\label{21-1-9}
\begin{split}
& \bigg(\frac N{2} - \frac {N-2\alpha}2 \bigg)\lambda\int_{\mathcal B_n\cap\{y=0\}} v_n^{2}\,dx \\
&\le  \int_{(\partial_i\mathcal B_n)\cap\{y=0\}}(\frac 1{p_n} |v_n|^{p_n} +\frac12\lambda v_n^2)\langle x-x_0,\nu_x\rangle\,dS\\
&+ \frac {N-2\alpha}2\int_{\partial_i \mathcal B_n}y^{1-2\alpha}v_n\frac{\partial v_n}{\partial\nu}\,dS\\
& -\frac 12 \int_{\partial_i \mathcal B_n}y^{1-2\alpha}|\nabla v_n|^2(X-z_0,\nu)\,dS\\
&+\int_{\partial_i \mathcal B_n}y^{1-2\alpha}(\nabla v_n, X-z_0)\frac{\partial v_n}{\partial\nu}\,dS.\\
\end{split}
\end{equation}

By \eqref{eq:5.3}, we find

\begin{equation}\label{30-1-9}
\begin{split}
\text{ RHS of \eqref{21-1-9}}\le
& C\sigma_n^{-\frac12} \int_{(\partial_i\mathcal B_n)\cap\{y=0\}}(|v_n|^{p_n} + v_n^2)
\,dS\\
&+ C\Bigl(\int_{\partial_i \mathcal B_n}y^{1-2\alpha}|\nabla v_n|^2\,dS\Bigr)^{\frac12}\Bigl(\int_{\partial_i \mathcal B_n}y^{1-2\alpha}v_n^2\,dS\Bigr)^{\frac12}\\
&+
  C\sigma_n^{-\frac12} \int_{\partial_i \mathcal B_n}y^{1-2\alpha}|\nabla v_n|^2\,dS\\
&\leq
C\Bigl(\sigma_n^{-\frac12} \sigma_n^{\frac12-\alpha}+ \sigma_n^{\frac14}\sigma_n^{-\frac14} +\sigma_n^{-\frac12}\sigma_n^{\frac12}
\Bigr)\sigma_n^{-\frac{N-2\alpha}2}\le C\sigma_n^{-\frac{N-2\alpha}2}.
\end{split}
\end{equation}

Inserting \eqref{30-1-9} into \eqref{21-1-9}, we obtain

\begin{equation}\label{eq:5.8}
\int_{\mathcal B_n\cap\{y=0\}} v_n^{2}\,dx \leq C\sigma_n^{-\frac{N-2\alpha}2}.
\end{equation}

Let us assume that $\sigma_n= \sigma_{n,1}$.
 Using \eqref{19-30-8}, similarly
to \cite{CPY}, we can deduce that if $N>4\alpha$, then

\begin{equation}\label{1-2-9}
\begin{split}
&\int_{\mathcal B_n\cap\{y=0\}} v_n^{2}\,dx\ge
\int_{\mathcal B_{\sigma_n^{-1}}((x_n,0))\cap\{y=0\}} v_n^{2}\,dx \\
\ge & \frac12\int_{\mathcal B_{\sigma_n^{-1}}((x_n,0))\cap\{y=0\}} |\rho_{x_n^1,\sigma_n^1}(W_1)|^2
+o\bigl(\sigma_n^{-2\alpha}\bigr)\\
=& \frac12\sigma_n^{-2\alpha}\int_{\mathcal B_1(0)\cap\{y=0\}} W_1^2
+o\bigl(\sigma_n^{-2\alpha}\bigr).
\end{split}
\end{equation}

Combining \eqref{eq:5.8}  and \eqref{1-2-9},  we are led to
\[
\sigma_n^{-2\alpha}\leq C\sigma_n^{-\frac{N-2\alpha}2}.
\]
This is a contradiction if $N> 6\alpha$. $\Box$

{\bf Proof of Theorem \ref{thm:1.2}.}  It is
standard to prove that Theorem \ref{thm:1.2} follows directly from Theorem \ref{thm:1.1}.  See \cite{CDS,CPY}.
 For the convenience of the readers, we follow \cite{CPY} to outline the proof.

 For any $k\in \mathbb{N}$,  define the ${\mathbb{Z}}_2$-homotopy
class $\mathcal{F}_k$ by
\[
\mathcal{F}_k=\left\{A;\,\,A\in H^1_{0,L}(\mathcal{C}_\Omega)\, \text{is
compact},\mathbb{Z}_2-\text{invariant, and}\,\,\gamma(A)\geq
k\right\},
\]
where the genus $\gamma(A)$ is smallest integer $m$, such that there
exists an odd map $\phi\in C(A, \mathbb{R}^m\setminus \{0\})$. For $k=1,2,\cdots$,
 we can define the  minimax value
 \begin{equation}\label{cv}
c_{k,\,\varepsilon}=\inf_{A\in\mathcal{F}_k}  \max_{u\in A} I_{\varepsilon}(u),
\end{equation}
where  $I_{\varepsilon}(u)$ is defined in \eqref{eq:1.4}. Then, $c_{k,\varepsilon}$ is a critical value of  $I_{\varepsilon}(u)$,  Thus there is
 $u_{k, \varepsilon}$ such that
$I_{\varepsilon}(u_{k,\varepsilon})=c_{k,\varepsilon}$  and $I'_{\varepsilon}(u_{k,\varepsilon})=0$.

 For any $k=1, \cdots $,  it is easy to show that $|c_{k,\,\varepsilon}|\le C_k$ for some $C_k>0$ which
 is independent of $\varepsilon$. Therefore, $u_{k, \varepsilon}$ is bounded in $H^1_{0,L}(\mathcal{C}_\Omega)$ for any fixed $k$. By
Theorem \ref{thm:1.1}, up to a subsequence, $u_{k, \varepsilon}\to u_k$ strongly in $H^1_{0,L}(\mathcal{C}_\Omega)$. So, $u_k$ satisfies
$I_{0}(u_{k})=c_{k}:=\lim_{\varepsilon \to 0} c_{k,\,\varepsilon}$  and $I'_{0}(u_{k})=0$.

We are now ready to show that $I_0(u)$ has infinitely many critical
points.  Note that $c_k$ is non-decreasing in $k$.  We distinguish
several cases.

(1)  Suppose that there are $1< k_1 < \cdots k_{i} < \cdots$,
satisfying

\[
c_{k_1}<\cdots< c_{k_i}<\cdots.
\]
Then, we are done.  So we assume in the sequel that for some
positive integer $m$, $ c_k= c$ for all $k\ge m$.

(2)  Suppose that for any $\delta>0$, $I_0(u)$ has a critical point
$u$ with $I_0(u)\in (c-\delta,c+\delta)$ and $I_0(u)\ne c$.  In this
case, we are done. So from now on we  assume that there exists a
$\delta>0$, such that  $I_0(u)$ has no critical point $u$ with
$I_0(u)\in (c-\delta,\,c)\bigcup(c,\,c+\delta)$. In this case, using
the deformation argument,    we can
 prove that
\begin{equation}\label{gg1}
\gamma(K_c)\ge 2,
\end{equation}
where $K_c= \bigl\{ u\in H^1_{0,L}(\mathcal{C}_\Omega): I'_0(u)=0,\;\;
I_0(u)=c\bigr\}$. As a consequence, $I_0(u)$ has infinitely many
critical points.
$\Box$

\appendix

\section{ Estimates for a linear problem}

In this section,  we will establish the $L^p$ estimates for a linear
problem. Let $D$ be any bounded domain in  $\mathbb{R}^N$. Recall that we use
the notations $ \mathcal C_D = D\times (0, +\infty)$  and
 $\partial_L \mathcal C_D =\partial D\times (0, +\infty)$ Consider

\begin{equation}\label{eq:2.15}
\left\{ \arraycolsep=1.5pt
\begin{array}{lll}
div(y^{1-2\alpha}\nabla w)=0 \;\;   &{\rm in}\;  \mathcal C_D  , \\[2mm]
 w=0  & {\rm on}\quad \partial_L \mathcal C_D, \\[2mm]
 y^{1-2\alpha}\frac{\partial w}{\partial \nu}=f(x) & {\rm on}\quad
 D\times\{0\}.\\[2mm]
 \end{array}
 \right.
 \end{equation}

\begin{Proposition}\label{lem:2.2}

Suppose that $f\in C^\beta(D), f\geq 0$. Let  $w$ be the solution of \eqref{eq:2.15}. Then for any $1<p<\frac
N{2\alpha}$, there is a constant $C>0$, such that
$$
\|w(\cdot,0)\|_{L^{\frac {Np}{N-2\alpha p}}(D)}\leq C\|f\|_{L^p(D)}.
$$

\end{Proposition}

\begin{proof}

First, it is easy to see that $w>0$.

We claim that if $q>\frac12$,  then

\begin{equation}\label{neq:2.6}
\Bigl(\int_{D }|
w^{q}(x,0)|^{2^*_\alpha}\,dx\Bigr)^{2/2^*_\alpha}\leq C\int_{ D}
f(x) w^{2q-1}(x,0)\,dx.
\end{equation}

Note that $w\in L^\infty( \mathcal C_D )$.

 We first assume $q\geq 1$.
Let $\varphi=w^{2q-1}\in H^1_{0,L}(\mathcal{C}_D)$. Testing
\eqref{eq:2.15} by $\varphi$, we obtain
$$
\int_{ \mathcal C_D }  y^{1-2\alpha}\nabla w\nabla
\varphi\,dxdy=\int_{D} f(x) w^{2q-1}(x,0)\,dx.
$$

 We deduce
\begin{eqnarray*}
&&\int_{ \mathcal C_D } y^{1-2\alpha} \nabla w\nabla \varphi\,dxdy\\
&= & \frac{2q-1}{q^2}\int_{ \mathcal C_D } y^{1-2\alpha}|\nabla  w^q|^2\,dxdy\\
&\ge & c_0(q)\Bigl(\int_{D }|
w^q(x,0)|^{2^*_\alpha}\,dx\Bigr)^{2/2^*_\alpha}
\end{eqnarray*}
where $ c_0(q)>0 $  is   some constant. Hence,
\begin{equation}\label{eq:2.6}
\Bigl(\int_{D }|
w^q(x,0)|^{2^*_\alpha}\,dx\Bigr)^{2/2^*_\alpha}\leq C\int_{ D} f(x)
w^{2q-1}(x,0)\,dx.
\end{equation}

Now we consider the case $q\in (\frac12, 1)$. For any $\theta>0$,
let $\eta = w (w+\theta)^{2(q-1)}\in H^1_{0,L}(\mathcal{C}_D)$. Then

\[
\nabla \eta = (w+\theta)^{2(q-1)} \nabla w+2(q-1) w
(w+\theta)^{2q-3}\nabla w
\]
From  $q\in (\frac12,1)$, we find

\[
\begin{split}
&\int_{ \mathcal C_D } y^{1-2\alpha}\nabla w \nabla \eta \\
\ge & (2q-1) \int_{ \mathcal C_D } y^{1-2\alpha}(w+\theta)^{2(q-1)}| \nabla w|^2\\
=& \frac{2q-1}{q^2} \int_{ \mathcal C_D } y^{1-2\alpha}|
\nabla ( ( w+\theta)^q-
\theta^q)|^2\\
\ge &  c_0(q)\Bigl(\int_{D }| ( w(x,0)+\theta)^q-
\theta^q|^{2^*_\alpha}\,dx\Bigr)^{2/2^*_\alpha}.
\end{split}
\]
So, we obtain

\begin{equation}\label{1-6-8}
\Bigl(\int_{D }| ( w(x,0)+\theta)^q-
\theta^q|^{2^*_\alpha}\,dx\Bigr)^{2/2^*_\alpha}\le C \int_{D}
f(x)w(x,0)( w(x,0)+\theta)^{2q-2}\,dx.
\end{equation}

Letting $\theta\to 0$ in \eqref{1-6-8}, we obtain \eqref{neq:2.6}.

On the other hand,

\begin{equation}\label{eq:2.7}
\int_{ D } f(x) w^{2q-1}(x,0)\,dx\le \Bigl(\int_{D }
|f|^{\frac{2^*_\alpha q} {2^*_\alpha q-2q+1}}\Bigr)^{\frac
{2^*_\alpha q-2q+1}{2^*_\alpha q}} \|w\|_{q 2^*_\alpha}^{2q-1},
\end{equation}
By \eqref{eq:2.6}, \eqref{eq:2.7} and the embedding
$H^1_{0,L}(\mathcal{C}_D)\hookrightarrow L^{2_\alpha^*}(\Omega)$,
which, together with \eqref{neq:2.6}, gives

\begin{equation}\label{2-6-8}
\|w\|_{q 2^*_\alpha}\le C \Bigl(\int_{D } |f|^{\frac{2^*_\alpha q}
{2^*_\alpha q-2q+1}}\,dx\Bigr)^{\frac {2^*_\alpha q-2q+1}{2^*_\alpha
q}}.
\end{equation}

Let  $ p= \frac{2^*_\alpha q} {2^*_\alpha q-2q+1}$.  Then  $  q=
\frac{p}{2^*_\alpha  -(2^*_\alpha-2)p}>\frac12$, and

\[
2^*_\alpha q=  \frac{2^*_\alpha p}{2^*_\alpha
-(2^*_\alpha-2)p}=\frac{N p}{N-2 p \alpha}.
\]
The proof is complete.
\end{proof}

Let $w\in H^1_{0,L}(\mathcal{C}_D)$ be a solution of
\begin{equation}\label{eq:2.4}
\left\{ \arraycolsep=1.5pt
\begin{array}{lll}
div(y^{1-2\alpha}\nabla w)=0 \;\;   &{\rm in}\; \mathcal C_D , \\[2mm]
 w=0  & {\rm on}\quad \partial_L  \mathcal C_D  , \\[2mm]
 y^{1-2\alpha}\frac{\partial w}{\partial \nu}=a(x)v  &  x\in D,\; y=0.\\[2mm]
 \end{array}
 \right.
 \end{equation}

\begin{Corollary}\label{lem:2.1}

 Suppose $a,v\in C^\beta(D), 0<\beta<1,$ are nonnegative functions. Then, for any $p>\frac N {N-2\alpha}$, there is a constant
$C=C(p)>0$, such that
\begin{equation}\label{eq:2.5}
\|w(\cdot,0)\|_{L^{p}(D)}\leq C\|a\|_{L^{\frac
N{2\alpha}}(D)}\|v\|_{L^{p}(D)}.
\end{equation}
\end{Corollary}

\begin{proof}

Let  $f(x)=  a v$.  For any $q>1$, it follows from
Proposition~\ref{lem:2.2} that

\begin{equation}\label{3-6-8}
\|w(\cdot,0)\|_{L^{\frac{Nq}{N-2\alpha q}}(D)}\le C \| a v\|_{L^q(D)}\le C
\|a\|_{L^{\frac N{2\alpha}}(D)}\|v\|_{ L^{\frac{Nq}{N-2\alpha
q}}(D)}
\end{equation}
We thus prove this corollary by letting  $p=\frac{Nq}{N-2\alpha q}$.

\end{proof}

\begin{Corollary}\label{lem:2.3}
Let $w\in H^1_{0,L}(\mathcal{C}_D)$ be a solution of \eqref{eq:2.4}
with $a, v\geq 0$ and $a, v\in C^\beta(D)$. Then for any $\frac N
{N-2\alpha}<p_2<\frac {2N}{N-2\alpha}$, there is a constant
$C=C(p_2)>0$ such that

\begin{equation}\label{eq:2.18}
\|w(\cdot,0)\|_{L^{p_2}(D)}\leq
C\|a\|_{L^r(D)}\|v\|_{L^{2^*_\alpha}(D)},
\end{equation}
where $\frac 1{p_2}=\frac 1r+\frac 1{2^*_\alpha}-\frac {2\alpha}N$.
\end{Corollary}

\begin{proof}

Similar to the proof of   Corollary~\ref{lem:2.1},  we have

\begin{equation}\label{10-6-8}
\|w(\cdot,0)\|_{L^{p_2}(D)}\leq C\|a v\|_{L^{\frac{p_2N}{N+2\alpha
p_2}}(D)}\le C \|a\|_{L^r(D)}\|v\|_{L^{2^*_\alpha}(D)},
\end{equation}
where $r$ is determined by

\[
\frac1 r= \frac{N+2\alpha p_2}{p_2N}-\frac1{2^*_\alpha}=\frac1{p_2}
+\frac{2\alpha}N -\frac1{2^*_\alpha}.
\]

\end{proof}

\section{Decay estimate }

Consider the following problem:

\begin{equation}\label{1-21-8}
\left\{ \arraycolsep=1.5pt
\begin{array}{lll}
div(y^{1-2\alpha}\nabla v) = 0, & {\rm in}\ \mathbb{R}^{N+1}_+,\\[2mm]
y^{1-2\alpha}\frac {\partial v}{\partial y} = -
|v(x,0)|^{2^*_\alpha-2}v(x,0),   & {\rm in}\ \ \mathbb{R}^N,
 \end{array}
 \right.
 \end{equation}

In this section,  we will obtain a decay estimate for solutions of
\eqref{1-21-8}.

\begin{Proposition}\label{p1-30} Suppose $v\in H^1_{0,L}(\mathbb{R}^{N+1}
)$ is a solution of \eqref{1-21-8},
then there exists $C >0$
such that
\begin{equation}\label{eq:2.3a}
|v(X)|\leq \frac C{(1+ |X|^2 )^{\frac{N-2\alpha}2}}
\end{equation}
for $ X\in \mathbb{R}^{N+1}_+$.
\end{Proposition}

Before we prove Proposition~\ref{p1-30}, we need the following
lemma.

\begin{Lemma}\label{l1-30}

For any $u \in C_0^\infty (\mathbb{R}^{N+1})$, it holds

\[
\int_{\mathbb{R}^{N+1}}|y|^{1-2\alpha}\frac{u^2}{|X|^2}\,dxdy\leq
C\int_{\mathbb{R}^{N+1}}|y|^{1-2\alpha}|\nabla u|^2\,dxdy.
\]

\end{Lemma}

\begin{proof}
This lemma may be known.  Since the proof is short, we give the proof
here.

 Let

\[
   V(X)= \frac{|y|^{1-2\alpha }}{(N-2\alpha) |X|^2}X.
   \]
Then

\[
div V = \frac{|y|^{1-2\alpha}}{|X|^2}.
   \]
Thus

\[
\begin{split}
& \int_{\mathbb{R}^{N+1}} \frac{|y|^{1-2\alpha} u^2 }{|X|^2}  =
\int_{\mathbb{R}^{N+1}} u^2 div V\\
 = & -\int_{\mathbb{R}^{N+1}} 2 u
\nabla u \cdot V = -2 \int_{\mathbb{R}^{N+1}} u \nabla u\cdot
\frac{|y|^{1-2\alpha} X}{(N-2\alpha)|X|^2}
\\
\le & \frac2{N-2\alpha}\Bigl(\int_{\mathbb{R}^{N+1}} |y|^{1-2\alpha}
|\nabla u|^2\Bigr)^{\frac12}
 \Bigl(\int_{\mathbb{R}^{N+1}} |y|^{1-2\alpha} \frac{u^2}{
 |X|^2}\Bigr)^{\frac12}.
 \end{split}
\]

 \end{proof}

\begin{proof}[Proof of Proposition~\ref{p1-30}]
To prove \eqref{eq:2.3a}, we use the following Kelvin transformation
\[
\tilde v(X) = |X|^{-N+2\alpha}v\bigg(\frac
X{|X|^2}\bigg)
\]
of $v$. If $v$ is a solution of \eqref{1-21-8}, then $\tilde v$
satisfies

\begin{equation}\label{1-30-8}
\left\{ \arraycolsep=1.5pt
\begin{array}{lll}
div(y^{1-2\alpha}\nabla \tilde v) = 0, & {\rm in}\ \mathbb{R}^{N+1}_+\setminus\{0\},\\[2mm]
y^{1-2\alpha}\frac {\partial \tilde v}{\partial y} = - |\tilde
v(x,0)|^{2^*_\alpha-2}\tilde v(x,0),   & {\rm in}\ \
\mathbb{R}^N\setminus\{0\},
 \end{array}
 \right.
 \end{equation}
 Moreover,  we have
\begin{equation}\label{10-30-8}
 \int_{\mathbb{R}^{N}}|\tilde v(x,0)|^{2^*_\alpha}\,dx\leq C.
\end{equation}

On the other hand, it follows from Lemma~\ref{l1-30} that

\begin{equation}\label{eq:2.3b}
\int_{ \mathbb{R}^{N+1}_+}y^{1-2\alpha}|\nabla \tilde v|^2\,dxdy\leq
C.
\end{equation}

From \eqref{10-30-8} and \eqref{eq:2.3b}, it is standard to prove that $\tilde v$ is a
solution of \eqref{1-21-8}.  Harnack inequality gives

\[
|\tilde v|\le C, \quad \text{in}\; \mathcal B_{1}(0)\cap
\mathbb{R}^{N+1}.
\]
Hence, \eqref{eq:2.3a} follows.

\end{proof}

\vspace{2mm}

\noindent{\bf Acknowledgment}\quad S.Yan is partially supported by
ARC, J. Yang is supported by NNSF of China, No:10961016; GAN PO 555 program of
Jiangxi and X.Yu is supported by NNSF of China, No:11101291.

\end{document}